\documentclass[12pt,oneside,reqno]{amsart}

\usepackage{amsmath}
\allowdisplaybreaks
\usepackage[sort]{natbib}
\usepackage{amsthm,amssymb,mathrsfs,mathtools,float, graphicx,appendix}
\usepackage{quiver,tikz-cd,tikz} 
\usepackage{caption}
\usepackage[dvipsnames,x11names,svgnames]{xcolor}
\usepackage{dsfont}
\usepackage{bbm}
\usepackage{mathrsfs}
\usepackage{mathdots}

\usepackage[colorlinks=true,linkcolor=DarkGreen,citecolor=DarkGreen, urlcolor=DarkGreen,pagebackref]{hyperref}

\usepackage{geometry}

\newcommand\X{\ensuremath{\mathbf{X}}}
\newcommand\Y{\ensuremath{\mathbf{Y}}}
\newcommand\Z{\ensuremath{\mathbf{Z}}}
\newcommand\W{\ensuremath{\mathbf{W}}}
\newcommand\V{\ensuremath{\mathbf{V}}}
\newcommand\U{\ensuremath{\mathbf{U}}}
\newcommand\SSS{\ensuremath{\mathbf{S}}}
\newcommand\TT{\ensuremath{\mathbf{T}}}
\newcommand\E{\ensuremath{\mathbf{E}}}
\newcommand\x{\ensuremath{\mathbf{x}}}
\newcommand\y{\ensuremath{\mathbf{y}}}
\newcommand\z{\ensuremath{\mathbf{z}}}
\newcommand\w{\ensuremath{\mathbf{w}}}
\newcommand\vv{\ensuremath{\mathbf{v}}}
\newcommand\uu{\ensuremath{\mathbf{u}}}

\newcommand{\indep}{\perp\!\!\!\perp}

\usepackage{tikz}
\usepackage{tikz-cd}
\newcommand{\dashedleftrightarrow}{\mathrel{\tikz[baseline=-0.5ex]{\draw[dashed,<->] (0,0)--(0.6,0);}}}

\usepackage{cleveref}

\newtheorem{theorem}{Theorem}[section]       
\newtheorem{lemma}[theorem]{Lemma}            
    
\newtheorem{proposition}[theorem]{Proposition}
  
\newtheorem{example}[theorem]{Example}        
      
\theoremstyle{definition}
\newtheorem{definition}[theorem]{Definition}  
\theoremstyle{remark}
\newtheorem{remark}[theorem]{Remark}

\begin{document}
 
\title{Generalization of Pearl's Front-Door Criterion}

\author{Carol Wu}
\email{\url{cwu88@student.ubc.ca}}
\address{Faculty of Science, University of British Columbia, 6200 University Blvd., Vancouver, BC V6T 1Z4, Canada}

\author{Elina Robeva}
\email{\url{erobeva@math.ubc.ca}}
\address{Department of Mathematics, University of British Columbia, 6200 University Blvd., Vancouver, BC V6T 1Z4, Canada}

\begin{abstract}
    Pearl's front-door criterion provides a set of sufficient conditions for estimating the total causal effect from observational data in the presence of latent confounding, using the functional $P(\y \mid \operatorname{do}(\x := \x^*)) = \sum_{\z} P(\z \mid \x^*) \sum_{\x} P(\y \mid \z, \x) P(\x)$. An open question is whether these conditions can be generalized to be both necessary and sufficient for the validity of this functional, similar to the generalization achieved for the back-door adjustment criterion by Shpitser \cite{SHPITSER10}. In this paper, we present a new, weakened set of graph-based conditions sufficient for the front-door formula to estimate the total causal effect, expanding the scope of problems amenable to front-door identification.
\end{abstract}

\maketitle
%\bogo{get rid of bayesian network definition}
\section{Introduction} \label{sec: intro}
    A fundamental problem in causal inference is determining when an interventional distribution, $P(\y \mid \operatorname{do}(\x := \x^*))$, can be identified from the observational distribution $P(\vv)$. While complete algorithms exist for identification \cite{SHPITSER10, TIKKA17}, they rely on having a fully specified causal graph, limiting their practical utility. Herein lies motivation for the development of general identification criteria—such as the back-door \cite{PEARL09}, front-door \cite{PEARL09}, and covariate adjustment criterion \cite{SHPITSER10}—which operate under coarser assumptions of the corresponding causal graph. These provide graph-based conditions for the application of specific identification functionals.

    For example, Pearl's back-door criterion gives sufficient conditions for identifying the causal effect $P(\y \mid \operatorname{do}(\x := \x^*))$ via the covariate adjustment formula
    \begin{align*}
        P(\y \mid \operatorname{do}(\x := \x^*)) = \sum_\z P(\y \mid \x^*, \z) P(\z) \tag{\textasteriskcentered} \label{eq: cov_adj}.  
    \end{align*}
    However, its conditions are not necessary for (\ref{eq: cov_adj}) to hold. This incompleteness was resolved by Shpitser \cite{SHPITSER10}, who generalized it to a sufficient \textit{and} necessary graphical criterion for covariate adjustment (\ref{eq: cov_adj}) to hold.

    Similarly, the front-door criterion justifies the front-door functional
    \begin{align*}
        P(\y \mid \operatorname{do}(\x := \x^*)) = \sum_{\z} P(\z \mid \x^*) \sum_{\x} P(\y \mid \z, \x) P(\x) \tag{\textasteriskcentered\textasteriskcentered} \label{eq: fd_formula}.
    \end{align*}
    % when the causal graph of $P$ contains a subset $\Z$ of observed vertices satisfying:
    % \begin{enumerate} 
    %     \item [(1)] $\Z$ blocks all directed paths from $\X$ to $\Y$.
    %     \item [(2)] There are no open back-door paths from $\X$ to $\Z$.
    %     \item [(3)] $\X$ blocks all back-door paths from $\Z$ to $\Y$.
    % \end{enumerate}
    In a DAG $G = (\V \sqcup \U, \E)$, the front-door criterion requires that the sets of variables $\X, \Y, \Z \subset \V$ satisfy:
    \begin{enumerate}
        \item [(1)] $\Z$ blocks all direct paths $\X$ to $\Y$.
        \item [(2)] There are no open back-door paths from $\X$ to $\Z$.
        \item [(3)] $\X$ blocks all back-door paths $\Z$ to $\Y$.
    \end{enumerate}
    The front-door criterion is also incomplete; there exist causal diagrams violating the front-door criterion, yet the front-door functional (\ref{eq: fd_formula}) 
    % $$P(\y \mid \operatorname{do}(\x := \x^*)) = \sum_\z P(\z \mid \x^*) \sum_\x P(\y \mid \z, \x) P(\x)$$ 
    holds (see \Cref{fig: non-examples_fdc}). In this work, we provide the following weakening of the front-door criterion, where the front-door functional remains valid.
    
    \begin{theorem} [Main Theorem] \label{thm: main}
        Let $\X, \Y, \Z \subset \V$ be disjoint. Suppose $\Z$ satisfies the following criterion relative to $(\X, \Y)$:
        \begin{enumerate}
            \item [(i)] There are no open back-door paths from $\X$ to $\Z$.
            \item [(ii)] There are no open, proper front-door paths from \(\X\) to \(\Y\) given \(\Z\).
        \end{enumerate}
        Then, the causal effect is determined by the front-door functional
        $$P(\y \mid \operatorname{do}(\x := \x^*)) = \sum_\z P(\z | \x^*) \sum_\x P(\y | \z,\x)P(\x).$$
    \end{theorem}

    Subsequent sections are organized as follows. \Cref{sec: preliminaries} discusses notation, assumptions, and background from causal inference, which familiar readers may choose to skip. \Cref{sec: main} presents our main results. The proof of \Cref{thm: main} is given in \Cref{apx: pf_main}. The necessity of Condition (ii) of \Cref{thm: main} is shown in \Cref{apx: pf_cond_ii_necessary}.

    \subsection*{Acknowledgments}
        This work was supported by an NSERC USRA awarded to C. W. under the supervision of E. Robeva, who holds a Canada CIFAR AI Chair and is supported by an NSERC Discovery Grant.

\section{Preliminaries} \label{sec: preliminaries}
    \subsection{Foundations and Problem Statement}
        This subsection establishes the graphical framework for causal inference and defines the central problem addressed in this work. 
        
        We assume familiarity with directed acyclic graphs (DAGs), d-separation, and interventional distributions. Unfamiliar readers may consult Appendix~\ref{apx: prelim prelim}, as the discussion in this subsection is intentionally brief and serves primarily to fix notation. 

        \subsubsection{Notation and Standing Assumptions}
            We use boldface for sets of variables (e.g., $\V, \X$) and standard font (e.g., $V, X$) for individual variables. $\U$ and $U$ refer to latent (unobserved) variables. We use shorthand $P_\X(\Y)$ for the causal effect of $\X$ on $\Y$. 

            \textbf{Standing Assumption.} Throughout, we assume all observed distributions $P(\vv)$ to be \emph{strictly positive}. That is, $P(\vv) > 0$ for all $\vv$ in the domain of $\V$. 
        \subsubsection{Causal Models}
            A \textit{directed acyclic graph (DAG)} is a graph $G = (\V, \E)$ with directed edges $V_1 \to V_2$ and no directed cycles. 

            We consider a DAG
            $G = (\V \sqcup \U, \E)$, where $\V$ denotes observed variables and $\U$ denotes latent variables. The joint distribution $P(\vv,\uu)$ factorizes according to $G$ if
            \[
                P(\vv, \uu) = \prod_{v \in \V \sqcup \U} P\big(v \mid Pa(v)\big).
            \]
            The observed distribution is obtained by marginalization,
            \[
            P(\vv) = \sum_{\uu} P(\vv,\uu).
            \]
        \subsubsection{Interventions}
            Let $\X, \Y \subseteq \V$. The \textit{interventional distribution} $\mathrm{do}(\X=\x^*)$ is defined by the truncated factorization
            \[
            P(\vv \mid \mathrm{do}(\X=\x^*)) 
            = \sum_{\uu}\prod_{v \in (\V \sqcup \U) \setminus \X} P(v \mid Pa(v)) \cdot \mathbbm{1}_{\x=\x^*}.
            \]

            The \textit{causal effect} of $\X$ on $\Y$ is given by the marginal
            \[
            P_\X(\y) := P(\y \mid \mathrm{do}(\X=\x^*)),
            \]
            and we will use this notation throughout.

        \subsubsection{The Identification Problem}
            A central question in causal inference is whether an interventional distribution $P_\X(\y)$ can be uniquely determined from the observational distribution $P(\vv)$ under a given graph. In the absence of latent confounding, the truncated factorization above provides an explicit expression for $P_\X(\y)$. However, when unobserved variables are present, such a representation may fail to exist. This leads to the definition of identifiability.
            
            \begin{definition} [Identifiability] \label{def: identifiability}
                Let $G = (\V \sqcup \U, \E)$ be a DAG. Let $\X, \Y \subseteq \V$ be disjoint.  
        
                The causal effect $P_\X(\y)$ is said to be \textit{identifiable} in $G$ if there exists a functional $f$ such that for any joint distribution $P(\vv, \uu)$ that factorizes according to $G$,
                \[
                P_\X(\y) = f\big(P(\vv)\big).
                \]
                
                Equivalently, $P_\X(\y)$ is uniquely determined by the observed distribution $P(\vv)$ and does not depend on the distribution of the unobserved variables $\U$.
            \end{definition}

            \begin{example} \label{ex: identifiability}
                % https://q.uiver.app/#q=WzAsOCxbMywwLCJYIl0sWzUsMCwiWSJdLFs0LDAsIloiXSxbMSwwLCJYIl0sWzIsMCwiWSJdLFs0LDEsIlxcdGV4dHsoYil9Il0sWzAsMCwiWiJdLFsxLDEsIihhKSJdLFswLDJdLFsyLDFdLFszLDRdLFszLDQsIlUiLDEseyJjdXJ2ZSI6LTMsInN0eWxlIjp7InRhaWwiOnsibmFtZSI6ImFycm93aGVhZCJ9LCJib2R5Ijp7Im5hbWUiOiJkYXNoZWQifX19XSxbMSwwLCJVIiwxLHsiY3VydmUiOjQsInN0eWxlIjp7InRhaWwiOnsibmFtZSI6ImFycm93aGVhZCJ9LCJib2R5Ijp7Im5hbWUiOiJkYXNoZWQifX19XSxbNiwzXV0=
                \[\begin{tikzcd}[ampersand replacement=\&]
                    Z \& X \& Y \& X \& Z \& Y \\
                    \& {(a)} \&\&\& {\text{(b)}}
                    \arrow[from=1-1, to=1-2]
                    \arrow[from=1-2, to=1-3]
                    \arrow["U"{description}, curve={height=-18pt}, dashed, tail reversed, from=1-2, to=1-3]
                    \arrow[from=1-4, to=1-5]
                    \arrow[from=1-5, to=1-6]
                    \arrow["U"{description}, curve={height=24pt}, dashed, tail reversed, from=1-6, to=1-4]
                \end{tikzcd}\]
                \captionof{figure}{} \label{fig: identifiability}

                \begin{enumerate}
                    \item [(a)] $P_X(y)$ is not identifiable. 
                    \item [(b)] $P_X(y) = \sum_z P(z \mid x^*) \sum_x P(y \mid z, x) P(x)$ (see \Cref{thm: fdc}).
                \end{enumerate}
            \end{example}
        \subsubsection{Identification Criteria}
            While complete identification algorithms exist for fully specified DAGs \cite{SHPITSER10, TIKKA17}, they are often impractical, as such detailed knowledge of a system's DAG is rarely available in practice. This motivates \textit{identification criteria}\textemdash graphical conditions which guarantee that a specific identifying functional, such as the one in Example~\ref{ex: identifiability}, is valid\textemdash which enables identification of causal effects under substantially coarser assumptions.

            The first of two main criteria is the back-door criterion. It formalizes the idea that the discrepancy between the causal effect $P_x(y)$ and the conditional distribution $P(y \mid x)$ is due to associations through back-door paths. If a set $\Z$ can block all such paths, the causal effect becomes identifiable.

            \begin{definition} [Back-door and Front-door Paths]
                A \textit{front-door} path from $X$ is one where the first edge points away from $X$. That is, $\pi: X \rightarrow \cdots$.
                
                All other paths are \textit{back-door} paths.
            \end{definition}

            \begin{definition} [Back-door Criterion (BDC)] \label{thm: bdc}
                Let $G = (\V \sqcup \U, \E)$. Let $\X, \Y, \Z \subset \V$ be disjoint. $\Z$ satisfies the \textit{back-door criterion} relative to $(\X, \Y)$ if 
                \begin{enumerate}
                    \item [(1)] $\Z$ blocks all back-door paths $\X$ to $\Y$.
                    \item [(2)] $\Z$ is not a descendant of $\X$.
                \end{enumerate}
                Then, $P_\X(\y)$ is identifiable by the covariate adjustment formula:
                \begin{align*}
                    P_\X(\y) = \sum_\z P(\y \mid \z, \x^*) P(\z) \tag{\textasteriskcentered}.
                \end{align*}
            \end{definition}
            However, the back-door criterion fails in the presence of unobserved confounding between $\X$ and $\Y$ (see \Cref{fig: identifiability}(a)). In these settings, the front-door criterion provides an alternative identification strategy exploiting the decomposition of the causal effect $\X \to \Y$ through a mediator $\Z$.
        \begin{theorem} [Pearl's Front-door Criterion (FDC)] \label{thm: fdc}
            Let $\X, \Y, \Z \subset \V$ be disjoint sets in $\V$. $\Z$ satisfies the \textit{front-door criterion} relative to $(\X, \Y)$ if 
            \begin{enumerate}
                \item [(1)] $\Z$ blocks all direct paths $\X$ to $\Y$.
                \item [(2)] There are no open back-door paths from $\X$ to $\Z$.
                \item [(3)] $\X$ blocks all back-door paths $\Z$ to $\Y$.
            \end{enumerate}
            Then, $P_\X(\y)$ is identified by the front-door functional:
            \[ P_\X(\y) = \sum_\z P(\z \mid \x^*) \sum_\x P(\y \mid \x, \z) P(\x).\]
        \end{theorem}
        
        It should be noted that both the BDC and FDC are sufficient but not necessary for identification via their respective functionals (see \cite{SHPITSER10} and \Cref{fig: non-examples_fdc} for examples). The incompleteness of the back-door criterion was resolved by Shpitser \cite{SHPITSER10}:
            \begin{definition} [Covariate Adjustment Criterion] \label{thm: gen bdc}
                $\Z$ satisfies the adjustment criterion relative to $(\X,\Y)$ in $G$ if
                \begin{enumerate}
                    \item [(1)] No element in $\Z$ is a descendant of any $W \not \in \X$ which lies on a direct path from $\X$ to $\Y$. I.e., there is no subgraph of the form
                    $$\begin{array}{ccccccccc}
                            X & \rightarrow & \hdots & \rightarrow & W & \rightarrow & \hdots & \rightarrow & Y \\
                            & & & & & \searrow & & & \\
                            & & & & & & \ddots & & \\
                            & & & & & & & \searrow & \\
                            & & & & & & & & Z \\
                    \end{array}$$
                    \item [(2)] All non-direct paths from $\X$ to $\Y$ are blocked by $\Z$.
                \end{enumerate}
                Covariate adjustment
                \begin{align*}
                    P_\X(\y) = \sum_\z P(\y \mid \z, \x^*) P(\z) \tag{\textasteriskcentered}.
                \end{align*}
                holds if and only if $\Z$ satisfies the adjustment criterion relative to $(\X,\Y)$ in $G$.
            \end{definition}
            
            \begin{remark}
                Covariate adjustment has also been studied in other causal models. A complete criterion was provided by Perkovi\'{c} et al.\ (2015) \cite{PERKOVIC2017} for directed acyclic graphs (DAGs), maximal ancestral graphs (MAGs), completed partially directed acyclic graphs (CPDAGs), and partial ancestral graphs (PAGs).
            \end{remark}
            Motivated by this complete characterization for the back-door case, we seek a generalization of the front-door criterion for DAGs. Our work relaxes the assumptions under which the front-door formula remains valid, providing a broader sufficient condition for the front-door functional. 
    \subsection{Technical Background and Tools}
        We now introduce the key technical tools used in our analysis: do-calculus, latent projections, and mixed graphical representations.  
        \subsubsection{Do-calculus}\label{sec: do-calc}
            Our primary tool for establishing identifiability will be Pearl's \textit{do-calculus} \cite{PEARL1995a}, a set of inference rules for manipulating interventional distributions $P_\X(\y)$ by checking d-separation conditions of subgraphs.
            
            \begin{theorem} [Rules of Do-Calculus] \label{thm: do-calc}
                Let $G = (\V \sqcup \U, \E)$, and let $\X, \Y, \Z, \W \subset \V$ be disjoint. Then, 
                \begin{enumerate}
                    \item [(1)] $P_\X(\y \mid \z, \w) = P_\X(\y \mid \w)$ if $(\Y \indep \Z \mid \W, \X)_{G_{\overline{\X}}}$
                    \item [(2)] $P_{\X,\Z}(\y\mid \w) = P_\X(\y \mid \z,\w)$ if $(\Y \indep \Z \mid \X,\W)_{G_{\overline{\X},\underline{\Z}}}$
                    \item [(3)] $P_{\X,\Z}(\y\mid \w) = P_{\X}(\y\mid \w) $ if $ (\Y \indep \Z \mid \X,\W)_{G_{\overline{\X}, \overline{\Z(\W)}}}$
                \end{enumerate}
                where $\Z(\W) = \Z \backslash An(\W)_{G_{\overline{\X}}}$ and ${G_{\overline{\X},\underline{\Z}}}$ is the graph obtained from $G$ by removing all incoming arrows to $\X$ and all outgoing arrows from $\Z$.
            \end{theorem} 

        Do-calculus is \textit{complete} for identifying interventional distributions in DAGs \cite{HUANG2006b} and ADMGs (see Definition~\ref{def: admg}) \cite{HUANG2006a}: if $P_\X(\y)$ is identifiable, there exists a finite sequence of do-calculus rules and standard probability manipulations that reduces it to an expression involving only observational (intervention-free) probabilities. This completeness underpins the Shpitser–Pearl ID algorithm, a sound and complete identification algorithm for ADMGs \cite{SHPITSER08}. This algorithm has been implemented in \textsf{R} \cite{TIKKA17}, and has been useful for verifying identifiability results.

        \subsubsection{Latent Projection}\label{apx: latent_projection}
            Working directly with DAGs containing many latent variables can be cumbersome. The operation of latent projection condenses such a DAG into a simpler mixed graph that preserves all relevant independence and identifiability relations among the observed variables.

            The construction of a latent projection was introduced by Verma \cite[Section~3.1.3]{VERMA1993}, but we will reference a graphically clearer formulation given by Tian and Pearl \cite[Definition~1]{TIAN2002}. 

            \begin{definition} [Direct Path]
                A path from $X$ to $Y$ is \textit{direct} if all edges along $\pi$ are directed and point away from the first node and towards the last node. I.e. $\pi: X \rightarrow \cdots \rightarrow Y$.
            \end{definition}

            \begin{definition} [Trek]
                A \emph{trek} from $X$ to $Y$ in a DAG is a path consisting of two direct paths with a common source $V$, one ending at $X$ and one at $Y$: $$\pi: X \leftarrow \cdots \leftarrow V \rightarrow \cdots \rightarrow Y.$$
            \end{definition}

            \begin{definition}[Acyclic Directed Mixed Graph (ADMG)]\label{def: admg}
                An \textit{acyclic directed mixed graph (ADMG)} is a graph $G = (\V, \E)$ containing directed edges $V_1 \to V_2$ and bidirected edges $V_1 \dashedleftrightarrow V_2$, with no directed cycles. A bidirected edge $V_1 \dashedleftrightarrow V_2$ represents the presence of an unobserved confounder $U$ with $V_1 \dashleftarrow U \dashrightarrow V_2$. 
            \end{definition}
            
            \begin{definition}[Latent Projection] \label{def: latent_projection}
                Let \( G = (\V \sqcup \U, \E) \) be a DAG. The \textit{latent projection} \( Pj(G, \V) \) of \( G \) onto \( \V \) is the ADMG obtained by:
                \begin{enumerate} 
                    \item Removing all latent nodes $\U$ and their adjacent edges.
                    \item Adding a bidirected edge $V_1 \dashedleftrightarrow V_2$ for every trek in $G$ from $V_1$ to $V_2$ whose interior nodes lie entirely in $\U$ (if not already present).
                    \item Adding a directed edge $V_1 \to V_2$ for every direct path from $V_1$ to $V_2$ whose interior nodes lie entirely in $\U$ (if not already present).
                \end{enumerate}
                Any edge added in (2) or (3) is said to be \textit{induced} by the corresponding trek or direct path.
                
                Paths in \(G\) project naturally to paths in \(Pj(G, \V)\) by replacing latent segments with the induced edges.
            \end{definition}
                
            \begin{example}[Latent Projection]
                % https://q.uiver.app/#q=WzAsMTQsWzAsMSwiWCJdLFsxLDEsIlpfMSJdLFsyLDEsIlpfMiJdLFszLDEsIlVfMSJdLFs0LDEsIlkiXSxbMSwyLCJVXzQiXSxbMiwwLCJVXzMiXSxbMiwzLCJVXzUiXSxbMiw0LCJcXHRleHR7KGEpfSJdLFs2LDQsIlxcdGV4dHsoYil9Il0sWzUsMSwiWCJdLFs2LDEsIlpfMSJdLFs3LDEsIlpfMiJdLFs4LDEsIlkiXSxbMCwxXSxbMSwyXSxbMiwzLCIiLDAseyJzdHlsZSI6eyJib2R5Ijp7Im5hbWUiOiJkYXNoZWQifX19XSxbMyw0LCIiLDAseyJzdHlsZSI6eyJib2R5Ijp7Im5hbWUiOiJkYXNoZWQifX19XSxbNSwwLCIiLDAseyJzdHlsZSI6eyJib2R5Ijp7Im5hbWUiOiJkYXNoZWQifX19XSxbNSwxLCIiLDAseyJzdHlsZSI6eyJib2R5Ijp7Im5hbWUiOiJkYXNoZWQifX19XSxbNiwxLCIiLDAseyJzdHlsZSI6eyJib2R5Ijp7Im5hbWUiOiJkYXNoZWQifX19XSxbNiw0LCIiLDAseyJzdHlsZSI6eyJib2R5Ijp7Im5hbWUiOiJkYXNoZWQifX19XSxbMiwzLCJVMiIsMSx7Im9mZnNldCI6LTEsImN1cnZlIjozLCJzdHlsZSI6eyJ0YWlsIjp7Im5hbWUiOiJhcnJvd2hlYWQifSwiYm9keSI6eyJuYW1lIjoiZGFzaGVkIn19fV0sWzcsNSwiIiwwLHsic3R5bGUiOnsiYm9keSI6eyJuYW1lIjoiZGFzaGVkIn19fV0sWzcsNCwiIiwwLHsic3R5bGUiOnsiYm9keSI6eyJuYW1lIjoiZGFzaGVkIn19fV0sWzEwLDExXSxbMTIsMTMsIiIsMCx7ImN1cnZlIjotMiwic3R5bGUiOnsidGFpbCI6eyJuYW1lIjoiYXJyb3doZWFkIn0sImJvZHkiOnsibmFtZSI6ImRhc2hlZCJ9fX1dLFsxMCwxMSwiIiwwLHsib2Zmc2V0IjotMSwiY3VydmUiOjIsInN0eWxlIjp7InRhaWwiOnsibmFtZSI6ImFycm93aGVhZCJ9LCJib2R5Ijp7Im5hbWUiOiJkYXNoZWQifX19XSxbMTAsMTMsIiIsMix7ImN1cnZlIjo1LCJzdHlsZSI6eyJ0YWlsIjp7Im5hbWUiOiJhcnJvd2hlYWQifSwiYm9keSI6eyJuYW1lIjoiZGFzaGVkIn19fV0sWzExLDEzLCIiLDIseyJvZmZzZXQiOi0xLCJjdXJ2ZSI6Mywic3R5bGUiOnsidGFpbCI6eyJuYW1lIjoiYXJyb3doZWFkIn0sImJvZHkiOnsibmFtZSI6ImRhc2hlZCJ9fX1dLFsxMiwxM10sWzExLDEyXV0=
                \[\begin{tikzcd}[ampersand replacement=\&]
                    \&\& {U_3} \\
                    X \& {Z_1} \& {Z_2} \& {U_1} \& Y \& X \& {Z_1} \& {Z_2} \& Y \\
                    \& {U_4} \\
                    \&\& {U_5} \\
                    \&\& {\text{(a)}} \&\&\&\& {\text{(b)}}
                    \arrow[dashed, from=1-3, to=2-2]
                    \arrow[dashed, from=1-3, to=2-5]
                    \arrow[from=2-1, to=2-2]
                    \arrow[from=2-2, to=2-3]
                    \arrow[dashed, from=2-3, to=2-4]
                    \arrow["U2"{description}, shift left, curve={height=18pt}, dashed, tail reversed, from=2-3, to=2-4]
                    \arrow[dashed, from=2-4, to=2-5]
                    \arrow[from=2-6, to=2-7]
                    \arrow[shift left, curve={height=12pt}, dashed, tail reversed, from=2-6, to=2-7]
                    \arrow[curve={height=30pt}, dashed, tail reversed, from=2-6, to=2-9]
                    \arrow[from=2-7, to=2-8]
                    \arrow[shift left, curve={height=18pt}, dashed, tail reversed, from=2-7, to=2-9]
                    \arrow[curve={height=-12pt}, dashed, tail reversed, from=2-8, to=2-9]
                    \arrow[from=2-8, to=2-9]
                    \arrow[dashed, from=3-2, to=2-1]
                    \arrow[dashed, from=3-2, to=2-2]
                    \arrow[dashed, from=4-3, to=2-5]
                    \arrow[dashed, from=4-3, to=3-2]
                \end{tikzcd}\]
                \captionof{figure}{(a) $G = (\V \sqcup \U, \E)$. (b) $Pj(G, \V)$.} \label{fig: latent_projection}
                
                \vspace{6.8pt}

                Above, the trek $X \dashleftarrow U_4 \dashleftarrow U_5 \dashrightarrow Y$ in $G$ induces $X \dashedleftrightarrow Y$ in $Pj(G, \V)$. The directed path $Z_2 \dashrightarrow U_1 \dashrightarrow Y$ induces the directed edge $Z_2 \rightarrow Y$. The path $\pi: X \dashrightarrow U_4 \dashrightarrow Z_1 \rightarrow Z_2 \dashrightarrow U_1 \dashrightarrow Y$ projects to the path $\pi': X \dashedleftrightarrow Z_1 \rightarrow Z_2 \rightarrow Y$.

            \end{example}
            \begin{lemma}[Properties of Latent Projection]\label{lem: latent_proj_properties}
                Let $G$ be a DAG and define \(G' := Pj(G, \V)\). Then,
                \begin{enumerate}
                    \item Paths in \(G\) correspond to projected paths in \(G'\) with the same sequence of observed vertices and arrow directions. 
                    \item All d-separation and d-connectedness relations amongst the variables $\V$ are preserved \cite{VERMA1993}.
                    \item A causal effect \(P_{\X}(\Y)\) is identifiable in \(G\) if and only if it is identifiable in \(G'\), with the same functional expression \cite{HUANG2006c}.
                \end{enumerate}
            \end{lemma}
            
            % \begin{lemma} \label{lem: latent_prog_preserves_identifiability}
            %     The causal effect $P_{\X}(\Y)$ is identifiable in $G$ if and only if it is identifiable in $Pj(G, \V)$. If they are identifiable, then they are so with the same functional expression for $P_{\X}(\Y)$ \cite{HUANG2006c}.
            % \end{lemma}
            Lemma~\ref{lem: latent_proj_properties} justifies projecting out variables—whether latent or observed—that are not relevant to a specific identifiability query, greatly simplifying graphical reasoning.
\section{Main results} \label{sec: main}
    A natural starting point in attempting to weaken the front-door criterion is to ask: which conditions of the original criterion are strictly necessary for the front-door functional to hold? We find that only the requirement that (1) $\Z$ blocks all directed paths from $\X$ to $\Y$ is indispensable. 
    \begin{theorem} [Condition (1) is necessary] \label{thm: cond_1_necessary}
        If there exists a direct path $\X$ to $\Y$ not blocked by $\Z$ in $G$, then there exists $P(\vv)$ which factorizes according to $G$, but
        \[
        P_\X(\y) \neq \sum_\z P(\z \mid \x^*) \sum_\x P(\y \mid \x, \z) P(\x)
        \]
    \end{theorem}
    \begin{proof}
        By hypothesis, let $$\pi: X \to \hdots \to Y$$ be a direct path from $X \in \X$ to $Y \in \Y$ not containing any node in $\Z$. Let $G_\pi := (\V \sqcup \U, \E')$ be the subgraph of $G = (\V \sqcup \U, \E)$ where $\E' = \operatorname{edge}(\pi)$. I.e., it is the subgraph where we remove all edges except those on path $\pi$. 

        It suffices to prove 
        \[
        P_\X(\y) \neq \sum_\z P(\z \mid \x^*) \sum_\x P(\y \mid \x, \z) P(\x)
        \]
        for $G_\pi$, as any distribution $P$ which factorizes according to $G_\pi$ will also factorize according to the denser graph $G$, though unfaithfully so. 

        We observe the following:
        \begin{enumerate}
            \item [(a)] There are no back-door paths $\X$ to $\Y$ in $G_\pi$ (the only path from $\X$ to $\Y$ is $\pi$).
            \item [(b)] $\Z \indep \Y, \X$, $\Z \indep \Y \mid \X$ (there are simply no paths connecting $\Z$ to $\X$ nor $\Y$, let alone open paths, since no node of $\Z$ lies on the path $\pi$).
        \end{enumerate}
        For any $P$ which factorizes according to $G_\pi$, applying Rule 2 of Pearl's do-calculus (see~\Cref{thm: do-calc} for the rules of Pearl's do-calc) with (a) grants us
        \[
        P_\X(\y) = P (\y \mid \x)
        \]
        Meanwhile,
        \begin{align*}
            \sum_\z P(\z \mid \x^*) \sum_\x P(\y \mid \x, \z) P(\x) &= \sum_\z P(\z) \sum_\x P(\y \mid \x) P(\x) \tag{By (b)}\\
            &= P(\y)
        \end{align*}
        It is not difficult to find a probability distribution $P(\vv)$ faithful to $G_\pi$, and hence where $\Y \not \indep \X$ so that $P(\y \mid \x) \neq P(\y)$.
    \end{proof}
    
    In contrast, items (2) and (3) are not necessary; one, either, or both can be relaxed without contradicting the front-door functional.
    
    Below are several examples of graphs which contradict conditions (2) and/or (3) whilst still satisfying the front-door functional. For selected examples, see Examples~\ref{ex: non-example fdc violate 2} and \ref{ex: non-example fdc violate 3}, where we include a full derivation to illustrate this.
        \textcolor{white}{:3}
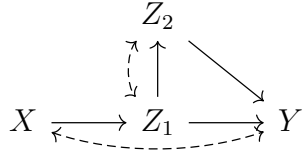
        \begin{center}
        \scalebox{0.57}{% https://q.uiver.app/#q=WzAsNDIsWzAsMiwiWF8xIl0sWzEsMiwiWF8yIl0sWzIsMiwiWSJdLFsxLDMsIloiXSxbMywzLCJYIl0sWzQsMywiWl8xIl0sWzUsMywiWl8yIl0sWzYsMywiWSJdLFs0LDIsIlpfMyJdLFs1LDIsIlpfNCJdLFsyLDEsIlgiXSxbMywxLCJaXzEiXSxbNCwxLCJZIl0sWzMsMCwiWl8yIl0sWzMsNCwiXFx0ZXh0eyhhKX0iXSxbOSw0LCJcXHRleHR7KGIpfSJdLFsxNCw0LCJcXHRleHR7KGMpfSJdLFsxNSwzLCJYIl0sWzE2LDMsIlpfMSJdLFsxNywzLCJZIl0sWzE2LDIsIlpfMyJdLFsxNSwyLCJaXzIiXSxbMTQsMCwiWl8yIl0sWzEzLDEsIlgiXSxbMTQsMSwiWl8xIl0sWzE1LDEsIlkiXSxbMTIsMywiWCJdLFsxMywzLCJaXzEiXSxbMTQsMywiWSJdLFsxMiwyLCJaXzIiXSxbMTMsMiwiWl8zIl0sWzE0LDIsIlpfNCJdLFs4LDMsIlhfMSJdLFs5LDMsIlpfMSJdLFsxMCwzLCJZIl0sWzksMiwiWF8yIl0sWzEwLDIsIlpfMyJdLFs4LDIsIlpfMiJdLFs5LDEsIlhfMiJdLFs4LDEsIlhfMSJdLFsxMCwxLCJaIl0sWzExLDEsIlkiXSxbMCwxXSxbMiwxXSxbMSwzXSxbNSw2XSxbNiw3XSxbOSw3LCIiLDAseyJzdHlsZSI6eyJ0YWlsIjp7Im5hbWUiOiJhcnJvd2hlYWQifSwiYm9keSI6eyJuYW1lIjoiZGFzaGVkIn19fV0sWzksOF0sWzgsNl0sWzQsNV0sWzQsOF0sWzcsNCwiIiwyLHsiY3VydmUiOi0zLCJzdHlsZSI6eyJ0YWlsIjp7Im5hbWUiOiJhcnJvd2hlYWQifSwiYm9keSI6eyJuYW1lIjoiZGFzaGVkIn19fV0sWzEwLDExXSxbMTEsMTJdLFsxMCwxMiwiIiwwLHsiY3VydmUiOjMsInN0eWxlIjp7InRhaWwiOnsibmFtZSI6ImFycm93aGVhZCJ9LCJib2R5Ijp7Im5hbWUiOiJkYXNoZWQifX19XSxbMTEsMTNdLFsxMywxMl0sWzEzLDExLCIiLDAseyJjdXJ2ZSI6Miwic3R5bGUiOnsidGFpbCI6eyJuYW1lIjoiYXJyb3doZWFkIn0sImJvZHkiOnsibmFtZSI6ImRhc2hlZCJ9fX1dLFsxOCwyMF0sWzIwLDE5LCIiLDEseyJzdHlsZSI6eyJ0YWlsIjp7Im5hbWUiOiJhcnJvd2hlYWQifSwiYm9keSI6eyJuYW1lIjoiZGFzaGVkIn19fV0sWzIxLDE3XSxbMjEsMjBdLFsyMiwyMywiIiwwLHsic3R5bGUiOnsidGFpbCI6eyJuYW1lIjoiYXJyb3doZWFkIn0sImJvZHkiOnsibmFtZSI6ImRhc2hlZCJ9fX1dLFsyMywyNF0sWzE4LDE5XSxbMTcsMThdLFsyNiwyN10sWzI2LDI5XSxbMjksMjddLFsyOCwzMV0sWzMwLDMxLCIiLDAseyJzdHlsZSI6eyJ0YWlsIjp7Im5hbWUiOiJhcnJvd2hlYWQifSwiYm9keSI6eyJuYW1lIjoiZGFzaGVkIn19fV0sWzMwLDI5XSxbMjcsMjhdLFsyNiwyOCwiIiwwLHsiY3VydmUiOjMsInN0eWxlIjp7InRhaWwiOnsibmFtZSI6ImFycm93aGVhZCJ9LCJib2R5Ijp7Im5hbWUiOiJkYXNoZWQifX19XSxbMzIsMzNdLFszMywzNF0sWzM1LDM2XSxbMzYsMzRdLFszNywzMiwiIiwxLHsic3R5bGUiOnsidGFpbCI6eyJuYW1lIjoiYXJyb3doZWFkIn0sImJvZHkiOnsibmFtZSI6ImRhc2hlZCJ9fX1dLFszOCwzOV0sWzM4LDQwXSxbNDAsNDFdLFszNSwzN10sWzI0LDI1XV0=
        \begin{tikzcd}[ampersand replacement=\&,column sep=scriptsize]
            \&\&\& {Z_2} \&\&\&\&\&\&\&\&\&\&\& {Z_2} \\
            \&\& X \& {Z_1} \& Y \&\&\&\& {X_1} \& {X_2} \& Z \& Y \&\& X \& {Z_1} \& Y \\
            {X_1} \& {X_2} \& Y \&\& {Z_3} \& {Z_4} \&\&\& {Z_2} \& {X_2} \& {Z_3} \&\& {Z_2} \& {Z_3} \& {Z_4} \& {Z_2} \& {Z_3} \\
            \& Z \&\& X \& {Z_1} \& {Z_2} \& Y \&\& {X_1} \& {Z_1} \& Y \&\& X \& {Z_1} \& Y \& X \& {Z_1} \& Y \\
            \&\&\& {\text{(a)}} \&\&\&\&\&\& {\text{(b)}} \&\&\&\&\& {\text{(c)}}
            \arrow[curve={height=12pt}, dashed, tail reversed, from=1-4, to=2-4]
            \arrow[from=1-4, to=2-5]
            \arrow[dashed, tail reversed, from=1-15, to=2-14]
            \arrow[from=2-3, to=2-4]
            \arrow[curve={height=18pt}, dashed, tail reversed, from=2-3, to=2-5]
            \arrow[from=2-4, to=1-4]
            \arrow[from=2-4, to=2-5]
            \arrow[from=2-10, to=2-9]
            \arrow[from=2-10, to=2-11]
            \arrow[from=2-11, to=2-12]
            \arrow[from=2-14, to=2-15]
            \arrow[from=2-15, to=2-16]
            \arrow[from=3-1, to=3-2]
            \arrow[from=3-2, to=4-2]
            \arrow[from=3-3, to=3-2]
            \arrow[from=3-5, to=4-6]
            \arrow[from=3-6, to=3-5]
            \arrow[dashed, tail reversed, from=3-6, to=4-7]
            \arrow[dashed, tail reversed, from=3-9, to=4-9]
            \arrow[from=3-10, to=3-9]
            \arrow[from=3-10, to=3-11]
            \arrow[from=3-11, to=4-11]
            \arrow[from=3-13, to=4-14]
            \arrow[from=3-14, to=3-13]
            \arrow[dashed, tail reversed, from=3-14, to=3-15]
            \arrow[from=3-16, to=3-17]
            \arrow[from=3-16, to=4-16]
            \arrow[dashed, tail reversed, from=3-17, to=4-18]
            \arrow[from=4-4, to=3-5]
            \arrow[from=4-4, to=4-5]
            \arrow[from=4-5, to=4-6]
            \arrow[from=4-6, to=4-7]
            \arrow[curve={height=-18pt}, dashed, tail reversed, from=4-7, to=4-4]
            \arrow[from=4-9, to=4-10]
            \arrow[from=4-10, to=4-11]
            \arrow[from=4-13, to=3-13]
            \arrow[from=4-13, to=4-14]
            \arrow[curve={height=18pt}, dashed, tail reversed, from=4-13, to=4-15]
            \arrow[from=4-14, to=4-15]
            \arrow[from=4-15, to=3-15]
            \arrow[from=4-16, to=4-17]
            \arrow[from=4-17, to=3-17]
            \arrow[from=4-17, to=4-18]
        \end{tikzcd}}
        \end{center}
        
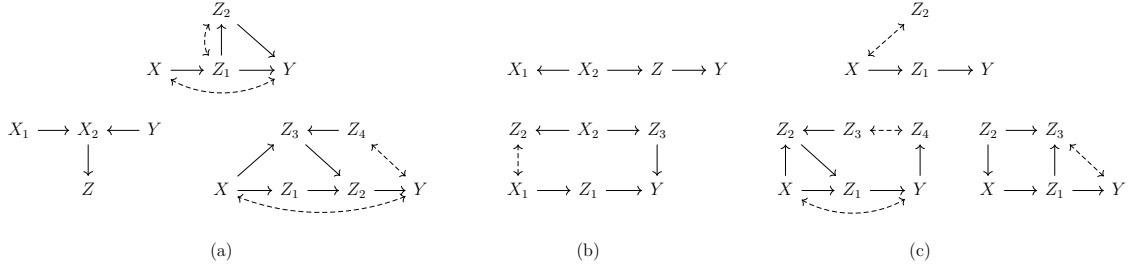
\captionof{figure}{Causal graphs where front-door functional $P_\X(\y) = \sum_\z P(\z \mid \x^*) \sum_\x P(\y \mid \x, \z) P(\x)$ is correct though (a) violate (3), (b) violate (2), (c) violate both.} \label{fig: non-examples_fdc}

    \begin{example} \label{ex: non-example fdc violate 3}
        % https://q.uiver.app/#q=WzAsNCxbMCwxLCJYIl0sWzEsMSwiWl8xIl0sWzIsMSwiWSJdLFsxLDAsIlpfMiJdLFswLDFdLFsxLDJdLFswLDIsIiIsMix7ImN1cnZlIjoyLCJzdHlsZSI6eyJ0YWlsIjp7Im5hbWUiOiJhcnJvd2hlYWQifSwiYm9keSI6eyJuYW1lIjoiZGFzaGVkIn19fV0sWzEsM10sWzMsMl0sWzMsMSwiIiwyLHsiY3VydmUiOjIsInN0eWxlIjp7InRhaWwiOnsibmFtZSI6ImFycm93aGVhZCJ9LCJib2R5Ijp7Im5hbWUiOiJkYXNoZWQifX19XV0=
        \[\begin{tikzcd}[ampersand replacement=\&]
            \& {Z_2} \\
            X \& {Z_1} \& Y
            \arrow[curve={height=12pt}, dashed, tail reversed, from=1-2, to=2-2]
            \arrow[from=1-2, to=2-3]
            \arrow[from=2-1, to=2-2]
            \arrow[curve={height=12pt}, dashed, tail reversed, from=2-1, to=2-3]
            \arrow[from=2-2, to=1-2]
            \arrow[from=2-2, to=2-3]
        \end{tikzcd}\]
        There exists a back-door path $Z_1 \dashedleftrightarrow Z_2 \rightarrow Y$ not blocked by $X$, violating (3). However, applying the Shpitser-Pearl ID algorithm (discussed in \Cref{sec: do-calc}) via its \textsf{R} implementation \cite{TIKKA17}, we obtain 
        \begin{align*}
            P_\x(\y) &= \sum_{z_1,z_2}P(z_2|x,z_1)P(z_1 \mid x)\left(\sum_{x}P(y\mid x,z_1,z_2)P(x)\right) \\
            &= \sum_{z_1,z_2}P(z_2,z_1\mid x)\left(\sum_{x}P(y\mid x,z_1,z_2)P(x)\right).
        \end{align*}
    \end{example}

    \begin{example} \label{ex: non-example fdc violate 2}
        % https://q.uiver.app/#q=WzAsNCxbMSwwLCJYXzIiXSxbMywwLCJZIl0sWzIsMCwiWiJdLFswLDAsIlhfMSJdLFswLDJdLFswLDNdLFsyLDFdXQ==
        \[\begin{tikzcd}[ampersand replacement=\&]
            {X_1} \& {X_2} \& Z \& Y
            \arrow[from=1-2, to=1-1]
            \arrow[from=1-2, to=1-3]
            \arrow[from=1-3, to=1-4]
        \end{tikzcd}\]

        There is an open back-door path $X_1 \leftarrow X_2 \rightarrow Z$, violating (2). However, the interventional distribution is still given by the front-door functional. 
        
        By application of Rule 3 and then Rule 2 of do-calc,
        \begin{align*}
            P_\x(\y) &= P_{x_2}(y) \\
            &= P (y \mid x_2).
        \end{align*}
        
        By simplifying the front-door functional, we obtain
        \begin{align*}
            &\sum_{z}P(z \mid x_1,x_2) \left(\sum_{x_1,x_2}P(y \mid x_1, x_2, z)P(x_1, x_2)\right) \\
            &= \sum_{z}P(z \mid x_1,x_2) \left(\sum_{x_1,x_2}P(y \mid z)P(x_1, x_2) \tag{$Y \indep \X \mid Z$}\right)\\
            &= \sum_{z}P(z \mid x_1,x_2) P(y \mid z) \\
            &= \sum_{z}P(z \mid x_2) P(y \mid z) \tag{$Z \indep X_1 \mid X_2$} \\
            &= \sum_{z}P(z \mid x_2) P(y \mid z, X_2) \tag{$Y \indep X_2 \mid Z$} \\
            &= \sum_{z} P(y, z \mid x_2) \\
            &= P(y \mid x_2) \\
            &= P_\x(\y),
        \end{align*}
        as desired. 
    \end{example}

    These examples demonstrate the original front-door criterion is sufficient but certainly not necessary for the interventional distribution to be given by the front-door functional. This motivates a search for a weaker set of graphical conditions. We now present a new criterion that strictly generalizes Pearl's.
        \begin{definition}
            A path $\pi$ from $\X$ to $\Y$ is \textit{proper} if the only node of $\pi$ in $\X$ is the first node.
        \end{definition}

        \begin{theorem} [Main Theorem] 
            Let $\X, \Y, \Z \subset \V$ be disjoint. Suppose $\Z$ satisfies the following criterion relative to $(\X, \Y)$:
            \begin{enumerate}
                \item [(i)] There are no open back-door paths from $\X$ to $\Z$.
                \item [(ii)] There are no open, proper front-door paths from \(\X\) to \(\Y\) given \(\Z\).
            \end{enumerate}
            Then, the causal effect is determined by the front-door functional
            $$P_\X(\y) = \sum_\z P(\z | \x^*) \sum_\x P(\y | \z,\x)P(\x)$$
        \end{theorem}
        The full proof, using Pearl's do-calculus, is provided in \Cref{apx: pf_main}.

        \begin{theorem}[Necessity of Condition (ii)] \label{thm: cond_ii_necessary}
        If there exists a proper, open front-door path $\X$ to $\Y$ conditioned on $\Z$, then there exists $P(\vv)$ which factorizes according to $G$, but the front-door functional does not hold. I.e.,
            \[
            P_\X(\y) \neq \sum_\z P(\z \mid \x^*) \sum_\x P(\y \mid \x, \z) P(\x)
            \]
        \end{theorem}
        
        The proof, given in \Cref{apx: pf_cond_ii_necessary}, proceeds by explicit construction of counterexamples. If condition (ii) is violated, then the latent projection $Pj(G, \X \sqcup \Y \sqcup \Z)$ necessarily contains one of the following subgraphs:
        \begin{enumerate}
            \item [(a)] $X \rightarrow Y$,
            \item [(b)] $X \rightarrow Z_1 \dashedleftrightarrow Z_2 \dashedleftrightarrow Z_3 \dashedleftrightarrow \cdots \dashedleftrightarrow Z_k \leftarrow Y$ ,
            \item [(c)] $X \rightarrow Z_1 \dashedleftrightarrow Z_2 \dashedleftrightarrow Z_3 \dashedleftrightarrow \cdots \dashedleftrightarrow Z_k \dashedleftrightarrow Y$.
        \end{enumerate}
        Each of these configurations admits a strictly positive distribution that factorizes according to $G$ but violates the front-door functional. 
        \begin{example}
            Consider case (b) with $k = 1$, yielding the graph
            % \begin{center}
            % % https://q.uiver.app/#q=WzAsMyxbMSwxLCJaXzEiXSxbMCwwLCJYPVVfMCJdLFsyLDAsIlkgPVVfMSJdLFsxLDBdLFsyLDBdXQ==
            % \begin{tikzcd}[ampersand replacement=\&]
            %     {X \sim \mathcal{U}(\{0,1\})} \&\& {Y \sim \mathcal{U}(\{0,1\})} \\
            %     \& {Z \sim \mathcal{U}(\{0,1, X, Y\})}
            %     \arrow[from=1-1, to=2-2]
            %     \arrow[from=1-3, to=2-2]
            % \end{tikzcd}.
            % \end{center}
            \[
            X \to Z \leftarrow Y.
            \]
            Let $X,Y \sim \mathcal{U}(\{0,1\})$, and let
            $Z \sim \mathcal{U}(\{0,1,X,Y\})$.
            This distribution is strictly positive and factorizes according to the graph.

            For $y = 0$, $x^* = 0$, the interventional distribution satisfies
            \[
            P_X(y) = \frac{1}{2}.
            \]
            % \begin{align*}
            %     P'_X(y) &= \sum_{z}P'(y = 0)P'(z \mid x^* = 0, y = 0) \tag{By definition} \\
            %     &= \frac{1}{2}.
            % \end{align*}
            However, direct computation shows that the front-door functional evaluates to
            % \begin{align*}
            %     &\sum_z P'(z \mid x^* = 0) \sum_{x} P'(y  = 0\mid z, x)P'(x) \\
            %     =& P'(z = 0 \mid x^* = 0) \left[ P'(y  = 0\mid z = 0, x = 0)P'(x = 0) +  P'(y  = 0\mid z = 0, x = 1)P'(x = 0)\right]\\
            %     +& P'(z = 1 \mid x^* = 0) \left[ P'(y  = 0\mid z = 1, x = 0)P'(x = 0) +  P'(y  = 0 \mid z = 1, x = 1)P'(x = 0)\right] \\
            %     =& \frac{5}{8}\cdot\left(\frac{3}{5}+\frac{2}{3}\right)\cdot\frac{1}{2} + \frac{3}{8}\cdot\left(\frac{1}{3} + \frac{2}{5}\right)\cdot\frac{1}{2} \\
            %     =& \frac{8}{15}.
            % \end{align*}
            \[
            \sum_z P(z \mid x^*) \sum_x P(y \mid z, x) P(x) = \frac{8}{15} \neq \frac{1}{2}.
            \]
            Thus, the front-door functional fails for this distribution.
        \end{example}
        We now summarize the implications of these results and their relationship to Pearl’s original front-door criterion.
    \subsection{Discussion and Implications} \label{subsec: discussion}
        Our criterion replaces the original conditions (2) and (3) with a weaker pair (i) and (ii). Condition (i) is identical to the original condition (2). Condition (ii) alone implies condition (1), but is overall strictly weaker than the original set of conditions (1) and (3); it is necessary (Theorem \ref{thm: cond_ii_necessary}) and, when combined with (i), sufficient.
        
        This generalization significantly expands the class of graphs for which the front-door functional is known to be valid. For instance, the models in Figures \ref{fig: non-examples_fdc}(a) and (b) are admitted by our new criterion but excluded by the original.
        
        For a brief motivation behind this criterion, we refer the interested reader to a heuristic deduction based on the theory of counterfactuals in Appendix \ref{apx: heuristic_derivations}. 
\section{Proof of \Cref{thm: main}} \label{apx: pf_main}
    \begin{proof} 
        \vspace{6.8pt}

        \noindent \textbf{Step 0:} To simplify the problem, we consider the latent projection
        \[
        G' = Pj(G, \X \sqcup \Y \sqcup \Z),
        \]
        which is the projection of \(G\) onto only the variables $\X, \Y, \Z$ relevant to the formula
        \[
        P_{\X}(\Y) = \sum_\z P(\z \mid \x^*) \sum_\x P(\y \mid \x, \z) P(\x).
        \]
        By completeness of do-calculus, if this formula holds in \(G'\), there exists a finite sequence of valid do-calculus operations in \(G'\) producing it.\footnote{In effect, the remaining observed nodes \(\W = \V \setminus (\X \sqcup \Y \sqcup \Z)\) of $G$ are treated as latent. While \(\W\) may influence the joint distribution in general, their values are unnecessary to compute \(P_\X(\Y)\); the latent projection ensures that all dependencies actually relevant to this computation among \(\X, \Y, \Z\) are preserved.} By Lemma~\ref{lem: latent_proj_properties}(1) and (2), all such operations remain valid in the original graph \(G\). Thus, it suffices to verify the formula on the ``smaller'' graph \(G'\), which is an ADMG with observed nodes \(\X, \Y, \Z\), and where Conditions (i) and (ii) continue to hold also by Lemma~\ref{lem: latent_proj_properties}(1) and (2).
        
        \vspace{6.8pt}
        
        \noindent \textbf{Step 1:} 
        We now walk through the algebraic derivation, citing the necessary conditions and rules from do-calculus (\Cref{thm: do-calc}) along the way. The validity of each step will be elaborated and proved in the subsequent discussion.
    
        We begin by decomposing $P_\X(\y)$ using marginalization:
        \begin{align*}
            P_\X(\y) &= \sum_\z P_\X(\y \mid \z) P_\X(\z) \tag{By marginalization}
        \end{align*}
    
        We need to evaluate the two factors $P_\X(\z)$ and $P_\X(\y \mid \z)$.
    
        For the first term, Condition (i) together with do-calculus Rule 2 imply
        \begin{align*}
            P_\X(\z) &= P(\z \mid \x^*) \tag{I: By Cond. (i) and Rule 2}
        \end{align*}
    
        For the second term, we partition $\Z$ into those variables which are children of $\X$ and those which are not, writing
        \[
            \Z^{ch} := \Z \cap Ch(\X), \quad \Z^{nch} = \Z \setminus \Z^{ch}.
        \]
        We can then expand and evaluate as follows:
        \begin{align*}
            P_\X (\y \mid \z) &= P_\X(\y \mid \z^{ch}, \z^{nch}) \\
            &= P_{\X, \Z^{ch}}(\y \mid \z^{nch}) \tag{II: By Cond. (ii) and Rule 2} \\
            &= P_{\Z^{ch}}(\y \mid \z^{nch}) \tag{III: By Cond. (ii) and Rule 3} \\
            &= \sum_\x P_{\z^{ch}}(\y \mid \x, \z^{nch}) P_{\Z^{ch}}(\x \mid \z^{nch}) \tag{By marginalization}\\
            P_{\Z^{ch}}(\x \mid \z^{nch}) &= P_{\Z^{ch}}(\x)\tag{IV: By Cond. (i), (ii) and Rule 1} \\
            &=P(\x) \tag{V: By Cond. (i) and Rule 3} \\
            P_{\Z^{ch}}(\y \mid \x, \z^{nch}) &= P(\y \mid \x, \z) \tag{VI: By Cond. (i), (ii) and Rule 2} \\
            P_\X (\y \mid \z) &= \sum_\x P(\y \mid \x, \z) P(\x)
        \end{align*}
        Then,
        \begin{align*}
            P_\X(\y) &= \sum_\z P(\z \mid \x^*) \sum_\x P(\y \mid \x, \z) P(\x)
        \end{align*}
        as desired.
        
        The list below discusses each application of do-calc above in greater detail.
            \begin{enumerate}
                \item [(I)] By rule 2 of do-calc, 
                $$P_\X(\z) = P(\z \mid \x^*) \text{ if } (\Z \indep \X)_{G_{\underline{\X}}}.$$
                The only paths from $\X$ to $\Z$ in $G_{\underline{\X}}$ must be back-door since all outgoing edges from $\X$ have been removed. By $(i)$, none of these paths can be open. Hence, $(\Z \indep \X)_{G_{\underline{\X}}}$ as desired.
                \item [(II)] By Rule 2 of do-calc,
                $$P_\X(\y \mid \z^{ch}, \z^{nch}) = P_{\X, \Z^{ch}}(\y \mid \z^{nch}) \text{ if } (\Y \indep \Z^{ch} \mid \X, \Z^{nch})_{G_{\overline{\X}\underline{\Z^{ch}}}}.$$
                If, for the sake of contradiction, $(\Y \not \indep \Z^{ch} \mid \X, \Z^{nch})_{G_{\overline{\X}\underline{\Z^{ch}}}}$, then there exists an open back-door path $\pi$ from $\Z^{ch}$ to $\Y$ given $\X, \Z^{nch}$ in $G_{\overline{X}\underline{\Z^{ch}}}$. If there exists an interior node of $\pi$ in $\Z^{ch}$, then the subpath of $\pi$ beginning at that interior node is also a back-door path (since there are no edges $\Z^{ch} \rightarrow \cdots$ from $\Z^{ch}$ in $G_{\overline{\X}\underline{\Z^{ch}}}$). Therefore, without loss of generality, assume the only node of $\pi$ in $\Z^{ch}$ is the first (i.e., $\pi$ is a proper path from $\Z^{ch}$ to $\Y$). Note $\pi$ contains no nodes in $\X$, as there are no edges going into $\X$ in any subgraph of $G_{\overline{X}}$, so any path containing $\X$ must contain it as a fork $\hdots \leftarrow X \rightarrow \hdots$, but then $\pi$ would be blocked since we condition on $\X$. Hence, $\pi$ is an open back-door path from $\Z^{ch}$ to $\Y$ conditioned on $\X, \Z^{nch}$ whose internal nodes are all in $\Z^{nch}$. Since $\X$ opens no colliders in any subgraph of $G_{\overline{X}}$ by virtue of having no ancestors nor incoming arrows, we may ignore conditioning on $\X$, and $\pi$ will remain open. From here, we can deduce every internal node, being in $\Z^{nch}$, is a collider. I.e., $\pi$ is of the form
                    $$\begin{array}{ccccccccccccccccccc}
                    Z^{ch} & \dashedleftrightarrow & Z^{nch} & \dashedleftrightarrow & Z^{nch} & \dashedleftrightarrow & \cdots & \dashedleftrightarrow & Y
                    \end{array}$$
                Since $Z^{ch}$ is a child of some $X \in \X$, if we append $X$ to the front of $\pi$, we have an open front-door path $\X$ to $\Y$ conditioned on $\Z$. 
                $$\begin{array}{ccccccccccccccccccc}
                    X & \rightarrow & Z^{ch} & \dashedleftrightarrow & Z^{nch} & \dashedleftrightarrow & Z^{nch} & \dashedleftrightarrow & \cdots & \dashedleftrightarrow & Y
                    \end{array}$$
                This contradicts condition (ii), and hence, $(\Y \indep \Z^{ch} \mid \X, \Z^{nch})_{G_{\overline{\X}\underline{\Z^{ch}}}}$, as desired.
                
                \item [(III)] By Rule 3 of do-calc, $$P_{\X, \Z^{ch}}(\y \mid \z^{nch}) = P_{\Z^{ch}}(\y \mid \z^{nch}) \text{ if } \left( \X \indep \Y \mid \Z \right)_{G_{\overline{\Z^{ch}}, \overline{\X(\Z^{nch})}}}$$ where $\X(\Z^{nch}) = \X \setminus An(\Z^{nch})_{G_{\overline{\Z^{ch}}}}$. 
                
                We claim in ${G_{\overline{\Z^{ch}}}}$, no $X \in \X$ is an ancestor of any $\Z$. If, for the sake of contradiction, $\X$ is an ancestor of some $\Z$ in ${G_{\overline{\Z^{ch}}}}$, then there exists a (without loss of generality, proper) direct path $\rho$ from some $X \in \X$ to some $Z \in \Z$. The second node of $\rho$ cannot be in $\Z$, as this would imply the existence of an edge $X \rightarrow Z^{ch}$ into $\Z^{ch}$, which does not exist in ${G_{\overline{\Z^{ch}}}}$. The second node of $\rho$ also cannot be in $\Y$, as this would violate condition (ii) which states there are no open proper front-door paths from $\X$ to $\Y$. Having exhausted all of our options, we conclude $\X$ and $An(\Z^{nch}) \subseteq An(\Z)$ share no common nodes. 
                
                Thus, $\X(\Z^{nch})$ simplifies to $\X$, and the condition for applying Rule 3 of do-calc reduces to \(\left( \X \indep \Y \mid \Z \right)_{G_{\overline{\Z^{ch}}, \overline{\X}}}\). Note (ii) implies precisely the independence \(\left( \X \indep \Y \mid \Z \right)_{G_{\overline{\X}}}\), and since removal of edges does not create new dependencies, it implies \(\left( \X \indep \Y \mid \Z \right)_{G_{\overline{\Z^{ch}}, \overline{\X}}}\), as desired.
                
                \item [(IV)] By Rule 1 of do-calc, $$P_{\Z^{ch}}(\x \mid \z^{nch}) = P_{\Z^{ch}}(\x) \text{ if } (\Z^{nch} \indep \X \mid \Z^{ch})_{G_{\overline{\Z^{ch}}}}.$$ Suppose for the sake of contradiction, there exists an open path $\pi$ from $\X$ to $\Z^{nch}$ given $\Z^{ch}$ in $G_{\overline{\Z^{ch}}}$. Without loss of generality, assume $\pi$ from $\X$ to $\Z^{nch}$ is a proper path (i.e. the only node in $\X$ is the first). Note conditioning on $\Z^{ch}$ cannot open any colliders in $G_{\overline{\Z^{ch}}}$ since all edges into $\Z^{ch}$ have been removed. Hence, conditioning on $\Z^{ch}$ only blocks formerly open paths. Hence, if $\pi$ is open when conditioning on $\Z^{ch}$, it remains open if we ignore conditioning. Without conditioning, we know from (i) there are no open back-door paths from $\X$ to $\Z$, so $\pi$ must instead be a front-door path. If so, then the second node must be either in $\Y$ or $\Z^{nch}$, since all edges into $\Z^{ch}$ have been removed. However, the second node cannot be $\Y$ since that is a direct (and hence front-door) path from $\X$ to $\Y$ not blocked by $\Z$, contradicting (ii); it also cannot be in $\Z^{nch}$ since that contradicts the definition of $\Z^{nch}$ being the set of non-children of $\X$. Hence, $\pi$ cannot exist, and thus $(\Z^{nch} \indep \X \mid \Z^{ch})_{G_{\overline{\Z^{ch}}}}$.
                \item [(V)] By Rule 3 of do-calc $$P_{\Z^{ch}}(\x) =P(\x) \text{ if } (\Z^{ch} \indep \X)_{G_{\overline{\Z^{ch}(\varnothing)}}}$$ 
                where $\Z^{ch}(\varnothing) = \Z^{ch} \setminus An(\varnothing) = \Z^{ch}$. If, on the contrary, there exists a path $\pi$ from $\Z^{ch}$ to $\X$ in $G_{\overline{\Z^{ch}}}$, it must be front-door since we have removed all incoming edges to $\Z^{ch}$. Then, since $\pi$ is front-door and we are not conditioning on anything, $\pi$ must be a direct path (from $\Z^{ch}$ to $\X$), but that would contradict (i) which states there are no back-door paths $\X$ to $\Z$. Hence, $(\Z^{ch} \indep \X)_{G_{\overline{\Z^{ch}(\varnothing)}}}$, as desired.
                \item [(VI)] By Rule 2 of do-calc, $$P_{\Z^{ch}}(\y \mid \x, \z^{nch}) = P(\y \mid \x, \z) \text{ if } (\Z^{ch} \indep \Y \mid \X, \Z^{nch})_{G_{\underline{\Z^{ch}}}}.$$
                Assume, for the sake of contradiction, the existence of a path $\pi$ from $\Z^{ch}$ to $\Y$ in $G_{\underline{\Z^{ch}}}$ given $\X, \Z^{nch}$. Without loss of generality, assume the only node in $\pi$ also in $\Z^{ch}$ is the first node, so that $\pi$ is proper. Then, all internal nodes in $\pi$ are in $\{\X, \Z^{nch}\}$, and since we are conditioning on $\{\X, \Z^{nch}\}$, they therefore must all be colliders. The path $\pi$ must be back-door since we are in $G_{\underline{\Z^{ch}}}$. We infer $\pi$ then must consist of a series of bidirected arcs
                $$\scalebox{0.78}{$\begin{array}{ccccccccccccccccccc}
                    Z^{ch} & \dashedleftrightarrow & Z^{nch} & \dashedleftrightarrow & Z^{nch} & \cdots & \dashedleftrightarrow & X & \dashedleftrightarrow & Z^{nch} & \dashedleftrightarrow & \cdots & \dashedleftrightarrow & Y
                    \end{array}$}$$
                or
                $$\scalebox{0.78}{$\begin{array}{ccccccccccccccccccc}
                    Z^{ch} & \dashedleftrightarrow & Z^{nch} & \dashedleftrightarrow & Z^{nch} & \cdots & \dashedleftrightarrow & X & \dashedleftrightarrow & Z^{nch} & \dashedleftrightarrow & \cdots & \leftarrow & Y
                    \end{array}$}$$
                However, $\pi$ cannot contain a node in $\X$; there would be a back-door path from $\X$ to $\Z$, violating condition (i). Hence, instead $\pi$ is of the form
                $$\scalebox{0.78}{$\begin{array}{ccccccccccccccccccc}
                    Z^{ch} & \dashedleftrightarrow & Z^{nch} & \dashedleftrightarrow & Z^{nch} & \dashedleftrightarrow & \cdots & \dashedleftrightarrow & Y
                    \end{array}$}$$
                or
                $$\scalebox{0.78}{$\begin{array}{ccccccccccccccccccc}
                    Z^{ch} & \dashedleftrightarrow & Z^{nch} & \dashedleftrightarrow & Z^{nch} & \dashedleftrightarrow & \cdots & \leftarrow & Y
                    \end{array}$}$$
                However, $Z^{ch}$ is a child of some $X \in \X$, so if we append $X$ to the start of $\pi$, 
                $$\scalebox{0.78}{$\begin{array}{ccccccccccccccccccc}
                    X & \rightarrow & Z^{ch} & \dashedleftrightarrow & Z^{nch} & \dashedleftrightarrow & Z^{nch} & \dashedleftrightarrow & \cdots & \dashedleftrightarrow & Y
                    \end{array}$}$$
                $$\scalebox{0.78}{$\begin{array}{ccccccccccccccccccc}
                    X & \rightarrow & Z^{ch} & \dashedleftrightarrow & Z^{nch} & \dashedleftrightarrow & Z^{nch} & \dashedleftrightarrow & \cdots & \leftarrow & Y
                    \end{array}$}$$
                we have an open, proper front-door path from $\X$ to $\Y$ conditioned on $\Z$, violating condition (ii). Hence, $\pi$ does not exist and $(\Z^{ch} \indep \Y \mid \X, \Z^{nch})_{G_{\underline{\Z^{ch}}}}$ as desired.
            \end{enumerate}
            Thus, if $\Z$ satisfies conditions (i) and (ii) relative to $\X, \Y$, then
            $$P_\X(\y) = \sum_\z P(\z | \x^*) \sum_\x P(\y | \z,\x)P(\x).$$
    \end{proof}

\section{Proof of \Cref{thm: cond_ii_necessary}} \label{apx: pf_cond_ii_necessary} 
    \begin{proof} 
        \noindent \textbf{Step 0:} We observe from the proof of $\Cref{thm: cond_1_necessary}$ it suffices to show the front-door formula fails on an appropriate subgraph of $G$. More formally,
        \begin{lemma} \label{lem: fdf_fails_in_subgraph_implies_fails_in_supgraph}
            Let $H = (\V' \sqcup \U', \E')$ be a subgraph of $G = (\V \sqcup \U, \E)$. Define $\X' := \X \cap \V'$, $\Y' := \Y \cap \V'$, and $\Z' := \Z \cap \V'$. Suppose $P'$ is a distribution which factorizes according to $H$. If
            \[
                P'_{\X'}(\y') \neq \sum_{\z'} P'(\z' \mid \x'^*) \sum_{\x'} P'(\y' \mid \z', \x') P'(\x'),
            \]
            then there exists a distribution $P$ which factorizes according to $G$ such that
            \[
                P_{\X}(\y) \neq \sum_{\z} P(\z \mid \x^*) \sum_{\x} P(\y \mid \z, \x) P(\x).
            \]
        \end{lemma}
        \begin{proof}
            Define $G_H := (\V \sqcup \U, \E')$, the graph obtained from $G$ by removing all edges not in $H$. To extend $P'$ to a distribution $P$ over $\V$, we set the value of every new variable $V \in \V \setminus \V'$ to be a fixed constant (e.g., 0), and defining $P$ to agree with $P'$ marginally over $V'$ (i.e., $P(\vv', \uu') := P'(\vv', \uu')$). By construction, $P$ factorizes according to $G_H$ (and hence $P$ also factorizes according to $G$, albeit possibly unfaithfully), and $P(\vv) = P'(\vv')$. 
            
            It is not too difficult to deduce $P_\X(\y) = P'_{\X'}(\y')$, and $\sum_{\z} P(\z \mid \x^*) \sum_{\x} P(\y \mid \z, \x) P(\x) = \sum_{\z'} P'(\z' \mid \x'^*) \sum_{\x'} P'(\y' \mid \z', \x') P'(\x')$.

            Then, $P'_{\X'}(\y') \neq \sum_{\z'} P'(\z' \mid \x'^*) \sum_{\x'} P'(\y' \mid \z', \x') P'(\x')$ implies $\sum_{\z} P(\z \mid \x^*) \sum_{\x} P(\y \mid \z, \x) P(\x) \neq P_\X(\y)$. 

            Thus, if the front-door formula fails on a subgraph of $G$, it fails on $G$.
        \end{proof}
    
        \begin{example}
            For $\Cref{thm: cond_1_necessary}$, the subgraph $H$ was the direct path $\pi: X \rightarrow \cdots \rightarrow Y$.    
        \end{example}
    
        Hence, we reduce the problem to finding subgraphs of $G$ such that the front-door formula fails.
        
        \vspace{6.8pt}
        
        \noindent \textbf{Step 1:}  
            We distinguish three cases:
            \begin{lemma} \label{lem: bad proj cases}
                Let $G' = Pj(G, \X \sqcup \Y \sqcup \Z)$. Condition (ii) (no proper front-door path $\X$ to $\Y$ conditioned on $\Z$) is violated in $G$ if and only if it is violated in $G'$, and if and only if $G'$ contains one of the following subgraphs:
                \begin{enumerate}
                    \item [(a)] $X \rightarrow Y$
                    \item [(b)] $X \rightarrow Z_1 \dashedleftrightarrow Z_2 \dashedleftrightarrow Z_3 \dashedleftrightarrow \cdots \dashedleftrightarrow Z_k \leftarrow Y$ 
                    \item [(c)] $X \rightarrow Z_1 \dashedleftrightarrow Z_2 \dashedleftrightarrow Z_3 \dashedleftrightarrow \cdots \dashedleftrightarrow Z_k \dashedleftrightarrow Y$ 
                \end{enumerate}
                where $k \geq 1$, and for some $X \in \X$, $Y \in \Y$, $Z_1,...,Z_k \in \Z$.
            \end{lemma}

            \begin{proof}
                The first part of the lemma, that (ii) is violated in $G$ if and only if it is violated in $G'$, follows from properties of latent projection (Lemma~\ref{lem: latent_proj_properties}).  
                
                The second part proceeds as follows: Suppose (ii) is violated in $G'$. Then there exists an open, proper, front-door path 
                $\pi$ from some $X \in \X$ to some $Y \in \Y$ given $\Z$. If there are no interior nodes, then we obtain graph (a). Otherwise, because $\pi$ is proper, no interior node of $\pi$ lies in $\X$, so all of its interior nodes instead lie in $\Z$. Then, because we are conditioning on $\Z$, every interior node of $\pi$ must be a collider. This forces the subpath through $\Z$ to consist solely of bidirected edges, with arrowheads meeting at each $Z_i$. The final interior node $Z_k$, as a collider, must connect to $Y$ either via an incoming edge $Z_k \leftarrow Y$ or a bidirected edge $Z_k \dashedleftrightarrow Y$. This yields (b) and (c) respectively.
    
                Conversely, the subgraphs (a), (b), (c) are clearly open, front-door paths from $\X$ to $\Y$ when conditioned on $\Z$. Therefore, if $G'$ contains such a subgraph, (ii) is violated.
            \end{proof}
    
            Case (a) is easily resolved as follows.
    
            \begin{lemma} \label{lem: x -> y implies dir path x to y}
                Let $G' = Pj(G, \X \sqcup \Y \sqcup \Z)$ be defined as in the previous lemma.  
                If $G'$ contains a subgraph of the form $X \to Y$ with $X \in \X$, $Y \in \Y$,  
                then $G$ contains a direct path from $\X$ to $\Y$ with no interior node in $\Z$.  
            \end{lemma}
            
            \begin{proof}
                This follows directly from the construction of latent projections (Definition~\ref{def: latent_projection}).
            \end{proof}

            \begin{lemma}
                Let $G' = Pj(G, \X \sqcup \Y \sqcup \Z)$ be defined as in the previous lemmas.  
                If $G'$ contains a subgraph of the form $X \to Y$ with $X \in \X$, $Y \in \Y$, then there exists a distribution $P(\vv)$ such that
                \[
                    P_\X(\y) \;\neq\; \sum_\z P(\z \mid \x^*) \sum_\x P(\y \mid \z,\x)\,P(\x).
                \]
            \end{lemma}
            \begin{proof}
                By \Cref{lem: x -> y implies dir path x to y}, $G$ violates condition~(1). The desired result follows from \Cref{thm: cond_1_necessary}.
            \end{proof}
            For cases (b) and (c), we proceed as follows. First, in Steps 2-6, we construct a probability distribution $P'$ on the subgraph (b) (or (c)) for which the front-door formula fails. Next, in Steps 7-10, we show that $P'$ can be ``lifted'' to a positive distribution $P$ on the pre-image of (b) (or (c)) in $G$. We may then conclude by Lemma~\ref{lem: fdf_fails_in_subgraph_implies_fails_in_supgraph} that the front-door formula fails for some extension of $P$ in the full graph $G$.
        \vspace{6.8pt}
        
        \noindent \textbf{Step 2: Counterexample model on (b).}
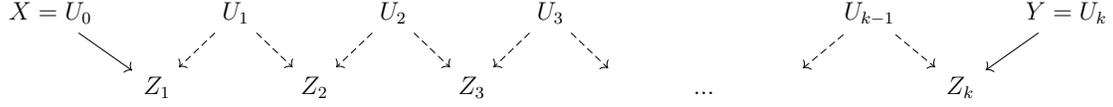
            \begin{center}
            \scalebox{0.8}{
            % https://q.uiver.app/#q=WzAsMTMsWzEsMSwiWl8xIl0sWzMsMSwiWl8yIl0sWzUsMSwiWl8zIl0sWzAsMCwiWD1VXzAiXSxbMiwwLCJVXzEiXSxbNCwwLCJVXzIiXSxbNiwwLCJVXzMiXSxbMTAsMCwiVV97ay0xfSJdLFsxMSwxLCJaX2siXSxbMTIsMCwiWSA9IFVfayJdLFs5LDEsIlxcdGV4dGNvbG9ye3doaXRlfXs6M30iXSxbOCwxLCIuLi4iXSxbNywxLCJcXHRleHRjb2xvcnt3aGl0ZX17OjN9Il0sWzMsMF0sWzQsMCwiIiwxLHsic3R5bGUiOnsiYm9keSI6eyJuYW1lIjoiZGFzaGVkIn19fV0sWzQsMSwiIiwxLHsic3R5bGUiOnsiYm9keSI6eyJuYW1lIjoiZGFzaGVkIn19fV0sWzUsMSwiIiwxLHsic3R5bGUiOnsiYm9keSI6eyJuYW1lIjoiZGFzaGVkIn19fV0sWzUsMiwiIiwxLHsic3R5bGUiOnsiYm9keSI6eyJuYW1lIjoiZGFzaGVkIn19fV0sWzYsMiwiIiwxLHsic3R5bGUiOnsiYm9keSI6eyJuYW1lIjoiZGFzaGVkIn19fV0sWzcsOCwiIiwxLHsic3R5bGUiOnsiYm9keSI6eyJuYW1lIjoiZGFzaGVkIn19fV0sWzksOF0sWzcsMTAsIiIsMSx7InN0eWxlIjp7ImJvZHkiOnsibmFtZSI6ImRhc2hlZCJ9fX1dLFs2LDEyLCIiLDEseyJzdHlsZSI6eyJib2R5Ijp7Im5hbWUiOiJkYXNoZWQifX19XV0=
            \begin{tikzcd}[ampersand replacement=\&,column sep=small,row sep=scriptsize]
                {X=U_0} \&\& {U_1} \&\& {U_2} \&\& {U_3} \&\&\&\& {U_{k-1}} \&\& {Y = U_k} \\
                \& {Z_1} \&\& {Z_2} \&\& {Z_3} \&\& {\textcolor{white}{:3}} \& {...} \& {\textcolor{white}{:3}} \&\& {Z_k}
                \arrow[from=1-1, to=2-2]
                \arrow[dashed, from=1-3, to=2-2]
                \arrow[dashed, from=1-3, to=2-4]
                \arrow[dashed, from=1-5, to=2-4]
                \arrow[dashed, from=1-5, to=2-6]
                \arrow[dashed, from=1-7, to=2-6]
                \arrow[dashed, from=1-7, to=2-8]
                \arrow[dashed, from=1-11, to=2-10]
                \arrow[dashed, from=1-11, to=2-12]
                \arrow[from=1-13, to=2-12]
            \end{tikzcd}
            }
            \end{center}
            Let $U_0, \ldots, U_k$ be i.i.d. fair coins $\mathcal{U}(\{0,1\})$, where $$X = U_0, Y= U_k$$
            and for each $i = 1, \ldots, k$ define
            $$Z_i \sim \text{ Uniform on the multiset $\left[0,1, U_{i-1}, U_i\right]$}.$$
            We define $P'$ to be the probability distribution over the random variables $X, Y, \Z$.
        \vspace{6.8pt}
        
        \noindent \textbf{Step 3: Matrix product representation of $P'$}:
            For $a, b \in \{0,1\}$
                $$P'(Z_i = 0 \mid U_{i-1} = a, U_{i} = b) = \frac{3 - a - b}{4}, \quad P'(Z_i = 1 \mid U_{i-1} = a, U_{i} = b) = \frac{1 + a + b}{4}.$$
            We define the following $2 \times 2$ ``transition matrices'' with rows indexed by $a$, columns indexed by $b$,
            \begin{align*}
                M_0 &:= 
                \begin{pmatrix}
                    P'(Z_i = 0 \mid U_{i-1} = 0, U_i = 0) & P'(Z_i = 0 \mid U_{i-1} = 0, U_i = 1) \\
                    P'(Z_i = 0 \mid U_{i-1} = 1, U_i = 0) & P'(Z_i = 0 \mid U_{i-1} = 1, U_i = 1)
                \end{pmatrix}
                = \begin{pmatrix}
                    \frac{3}{4} & \frac{1}{2} \\
                    \frac{1}{2} & \frac{1}{4}
                \end{pmatrix} \\
                M_1 &:= 
                \begin{pmatrix}
                    P'(Z_i = 1 \mid U_{i-1} = 0, U_i = 0) & P'(Z_i = 1 \mid U_{i-1} = 0, U_i = 1) \\
                    P'(Z_i = 1 \mid U_{i-1} = 1, U_i = 0) & P'(Z_i = 1 \mid U_{i-1} = 1, U_i = 1)
                \end{pmatrix}
                = \begin{pmatrix}
                    \frac{1}{4} & \frac{1}{2} \\
                    \frac{1}{2} & \frac{3}{4}
                \end{pmatrix}
            \end{align*}
            with determinants $\det(M_0) = \det(M_1) = -\frac{1}{16}$.
            
            For a k-tuple $\z = (z_1, \ldots, z_{k})$, we define the product
                $$A(\z) = M_{z_1} \cdots M_{z_k}$$
            and let $e_0 = \left(\substack{1 \\ 0}\right)$, $e_1 = \left(\substack{0 \\ 1}\right)$, $\mathbf{1} = \left(\substack{1 \\ 1}\right)$.
        
            \begin{lemma} The probability distribution $P'$ may be expressed as
                $$P'(X = a, Y = b, \Z = \z) = 2^{-(k+1)}e_a^T A(\z)e_b.$$
            \end{lemma}
            \begin{proof}
                Expand over $U_0, \ldots, U_k$ and use induction over $k$.
            \end{proof}
        \vspace{6.8pt}
        
        \noindent \textbf{Step 4: Target difference.} 
            Arbitrarily, we choose to verify $$P'_X(y) \neq \sum_\z P'(\z \mid \x^*) \sum_\x P'(\y \mid \x, \z) P(\x)$$ for the values $x^* = 0, y = 0$. We define the front-door expression at these values by
            $$S := \sum_\z P'(\z \mid x^* = 0) \sum_{x} P'(y  = 0\mid \z, x)P'(x)$$
            
            We define the shorthand
            $$r_a(\z) := P'(\z \mid x = a) = 2^{-k}e_a^TA(\z)\mathbf{1}, \quad q_a (\z) := P'(y = 0 \mid \z, x = a) = \frac{e_a^TA(\z)e_0}{e_a^TA(\z)\mathbf{1}}.$$
            Since $P'(x) = \frac{1}{2}$ is fair, we may factor it out of $S$. Substituting the definitions of $r_a$ and $q_a$, we obtain
            \begin{align*}
                S = \frac{1}{2}\sum_\z r_0(\z) \left(q_0(\z) + q_1(\z)\right)
            \end{align*}

            Meanwhile, by Rule 3 of do-calc and a short expansion,
            \begin{align*}
                P'_X(y) &= P'(y = 0) \\
                &= \frac{1}{2}\sum_\z \left(r_0(\z)q_0(\z) + r_1(\z)q_1(\z)\right).
            \end{align*}
            Subtracting,
            \begin{align*}
                S - P'_X(y) &= \frac{1}{2}\sum_\z\left(r_0(\z)q_1(\z) - r_1(\z)q_1(\z) \right) 
            \end{align*}
            The goal is to show this difference is non-zero.
        \vspace{6.8pt}
        
        \noindent \textbf{Step 5: Bit-flip symmetries and pairings.} 
            Let $\overline{\z}$ denote the inverse of $\z$ (i.e., $\z$ with every bit flipped). By the symmetries of our model,
            $$r_0(\overline{\z}) = r_1(\z), \quad q_0(\overline{\z}) = 1 - q_1(\z)$$
            Then, by pairing each $\z$ with $\overline{\z}$ and applying these symmetries,
            \begin{align*}
                S - P'_X(y) &= \frac{1}{4}\sum_\z \left(r_0(\z)q_1(\z) - r_1(\z)q_1(\z) + r_0(\overline{\z})q_1(\overline{\z}) - r_1(\overline{\z})q_1(\overline{\z}) \right) \\
                &= \frac{1}{4} \sum_\z \left(r_0(\z) - r_1(\z)\right) \left(q_0(\z) + q_1(\z) - 1\right)
            \end{align*}
            Define $E(\z) = \left(r_0(\z) - r_1(\z)\right) \left(q_0(\z) + q_1(\z) - 1\right)$. Then, rewrite
            $$S - P'_X(y) = \frac{1}{4} \sum_\z E(\z).$$
            We conduct a second pairing, this time by flipping only the last bit. For $\z = \left(z_1, ..., z_{k-1}, 0 \right)$, let $\z' = \left(z_1, ..., z_{k-1}, 1 \right)$. 

            We define $$A := M_{z_1} \cdots M_{z_{k-1}} = \begin{pmatrix} a_{00} & a_{01} \\ a_{10} & a_{11}\end{pmatrix}.$$ Then, $A(\z) = AM_0, A(\z') = AM_1$. By direct computation,
            \begin{align*}
                r_0(\z) &= 2^{-k}e_0^T(AM_0)\mathbf{1} = 2^{-(k + 2)}(5a_{00} + 3a_{01}) \\
                r_1(\z) &= 2^{-k}e_1^T(AM_0)\mathbf{1} = 2^{-(k + 2)}(5a_{10} + 3a_{11}) \\
                r_0(\z') &= 2^{-k}e_0^T(AM_1)\mathbf{1} = 2^{-(k + 2)}(3a_{00} + 5a_{01}) \\
                r_1(\z') &= 2^{-k}e_1^T(AM_1)\mathbf{1} = 2^{-(k + 2)}(3a_{10} + 5a_{11}) \\
                q_0(\z) &= \frac{e_0^T(AM_0)e_0}{e_0^T(AM_0)\mathbf{1}} = \frac{3a_{00} + 2a_{01}}{5a_{00} + 3a_{01}} \\
                q_1(\z) &= \frac{e_1^T(AM_0)e_0}{e_1^T(AM_0)\mathbf{1}} = \frac{3a_{10} + 2a_{11}}{5a_{10} + 3a_{11}} \\
                q_0(\z') &= \frac{e_0^T(AM_1)e_0}{e_0^T(AM_0)\mathbf{1}} = \frac{a_{00} + 2a_{01}}{3a_{00} + 5a_{01}} \\
                q_1(\z') &= \frac{e_1^T(AM_1)e_0}{e_1^T(AM_0)\mathbf{1}} = \frac{a_{10} + 2a_{11}}{3a_{10} + 5a_{11}} 
            \end{align*}
            %https://www.wolframalpha.com/input?i2d=true&i=%5C%2840%295%5C%2840%29a-c%5C%2841%29%2B3%5C%2840%29b-d%5C%2841%29%5C%2841%29+*+%5C%2840%295ac%2B4bc%2B4ad%2B3bd%5C%2841%29+Divide%5B1%2C+%5C%2840%295a%2B3b%5C%2841%29%5C%2840%295c+%2B+3d%5C%2841%29%5D++-+%5C%2840%295%5C%2840%29b-d%5C%2841%29%2B3%5C%2840%29a-c%5C%2841%29%5C%2841%29+*+%5C%2840%295bd%2B4bc%2B4ad%2B3ac%5C%2841%29+Divide%5B1%2C%5C%2840%293a%2B5b%5C%2841%29%5C%2840%293c+%2B+5d%5C%2841%29%5D
            By the power invested in WolframAlpha, the pair sum is of the form
            $$E(\z) + E(\z') = 2^{-k+1} \det (A) \Pi(A)$$
            where 
            \begin{align*}
                \Pi(A) =& \frac{1}{(5 a_{00} + 3 a_{01}) (3 a_{00} + 5 a_{01}) (5 a_{10} + 3 a_{11}) (3 a_{10} + 5 a_{11})} \\ &\cdot \bigg(15 a_{00}^2 a_{10} + 15 a_{00}^2 a_{11} + 34 a_{00} a_{01} a_{10} + 34 a_{00} a_{01} a_{11}+ 15 a_{00} a_{10}^2 + 34 a_{00} a_{10} a_{11} \\ &+ 15 a_{00} a_{11}^2 + 15 a_{01}^2 a_{10} + 15 a_{01}^2 a_{11} + 15 a_{01} a_{10}^2 + 34 a_{01} a_{10} a_{11} + 15 a_{01} a_{11}^2\bigg)
            \end{align*}
            Since every coefficient in $\Pi(A)$ is positive and $A$ itself has positive entries (being a product of matrices with positive entries), thus $\Pi(A) > 0$.

            By multiplicativity of determinants,
            $$\det (A) = \prod_{i=1}^{k-1}\det(M_{z_i}) = \left(-\frac{1}{16}\right)^{k-1}$$
            and hence $$\operatorname{sgn} \det(A) = (-1)^{k-1} \quad \operatorname{sgn} \left(E(\z) + E(\z')\right) = (-1)^{k-1}.$$ 
            Therefore,
            \begin{align*}
                S - P'_X(y) &= \frac{1}{4}\sum_\z E(\z) \\
                &= \frac{1}{4} \sum_{\text{pairs }(\z,\z')} \left(E(\z) + E(\z')\right)\\
                &= (-1)^{k-1} \cdot \text{ positive constant} \\
                &\neq 0    
            \end{align*}
            as desired.
            
        \vspace{6.8pt}
        \noindent \textbf{Step 6: Counterexample model on (c) from (b).} 
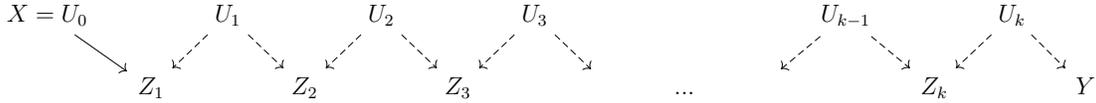
            \begin{center}
            \scalebox{0.78}{
            % https://q.uiver.app/#q=WzAsMTQsWzEsMSwiWl8xIl0sWzMsMSwiWl8yIl0sWzUsMSwiWl8zIl0sWzAsMCwiWD1VXzAiXSxbMiwwLCJVXzEiXSxbNCwwLCJVXzIiXSxbNiwwLCJVXzMiXSxbMTAsMCwiVV97ay0xfSJdLFsxMSwxLCJaX2siXSxbMTIsMCwiVV9rIl0sWzksMSwiXFx0ZXh0Y29sb3J7d2hpdGV9ezozfSJdLFs4LDEsIi4uLiJdLFs3LDEsIlxcdGV4dGNvbG9ye3doaXRlfXs6M30iXSxbMTMsMSwiWSJdLFszLDBdLFs0LDAsIiIsMSx7InN0eWxlIjp7ImJvZHkiOnsibmFtZSI6ImRhc2hlZCJ9fX1dLFs0LDEsIiIsMSx7InN0eWxlIjp7ImJvZHkiOnsibmFtZSI6ImRhc2hlZCJ9fX1dLFs1LDEsIiIsMSx7InN0eWxlIjp7ImJvZHkiOnsibmFtZSI6ImRhc2hlZCJ9fX1dLFs1LDIsIiIsMSx7InN0eWxlIjp7ImJvZHkiOnsibmFtZSI6ImRhc2hlZCJ9fX1dLFs2LDIsIiIsMSx7InN0eWxlIjp7ImJvZHkiOnsibmFtZSI6ImRhc2hlZCJ9fX1dLFs3LDgsIiIsMSx7InN0eWxlIjp7ImJvZHkiOnsibmFtZSI6ImRhc2hlZCJ9fX1dLFs5LDgsIiIsMSx7InN0eWxlIjp7ImJvZHkiOnsibmFtZSI6ImRhc2hlZCJ9fX1dLFs3LDEwLCIiLDEseyJzdHlsZSI6eyJib2R5Ijp7Im5hbWUiOiJkYXNoZWQifX19XSxbNiwxMiwiIiwxLHsic3R5bGUiOnsiYm9keSI6eyJuYW1lIjoiZGFzaGVkIn19fV0sWzksMTMsIiIsMSx7InN0eWxlIjp7ImJvZHkiOnsibmFtZSI6ImRhc2hlZCJ9fX1dXQ==
            \begin{tikzcd}[ampersand replacement=\&,column sep=small,row sep=scriptsize]
                {X=U_0} \&\& {U_1} \&\& {U_2} \&\& {U_3} \&\&\&\& {U_{k-1}} \&\& {U_k} \\
                \& {Z_1} \&\& {Z_2} \&\& {Z_3} \&\& {\textcolor{white}{:3}} \& {...} \& {\textcolor{white}{:3}} \&\& {Z_k} \&\& Y
                \arrow[from=1-1, to=2-2]
                \arrow[dashed, from=1-3, to=2-2]
                \arrow[dashed, from=1-3, to=2-4]
                \arrow[dashed, from=1-5, to=2-4]
                \arrow[dashed, from=1-5, to=2-6]
                \arrow[dashed, from=1-7, to=2-6]
                \arrow[dashed, from=1-7, to=2-8]
                \arrow[dashed, from=1-11, to=2-10]
                \arrow[dashed, from=1-11, to=2-12]
                \arrow[dashed, from=1-13, to=2-12]
                \arrow[dashed, from=1-13, to=2-14]
            \end{tikzcd}
            }
            \end{center}
            Let $U_0, ..., U_k \sim \mathcal{U}(\{0,1\})$, $Z_i \sim \mathcal{U}(\left[0, 1, Pa(Z_i)\right])$ as in the model for (b). Let $Y:= U_k$. Then, the observed distribution $P'(x, y, z_1,...,z_k)$ and interventional distribution $P'_\X(\y) = P'(y)$ are as they were in (b). Therefore, the front-door formula fails in (c).
        \vspace{6.8pt}
        
        \noindent\textbf{Interlude: Idea of Steps 7--9.}
            Lemma~\ref{lem: bad proj cases} states (ii) is violated if and only if the latent projection of $G$ onto $\X \sqcup \Y \sqcup\Z$, denoted $G'$, contains a subgraph of the form (a), (b), or (c). 

            The graph $G$ itself does not necessarily contain one of (a), (b), or (c). Only its projection is required to. Thus, it remains to show that the distribution $P'$ constructed on $G'$ can be lifted to a \emph{positive} distribution $P$ on $G$. 
            
            The following lemma describes the general form of a pre-image in $G$ of a subgraph of type (b). An analogous lemma for subgraphs of type (c) can be found in Step 10.
            \begin{lemma}\label{lem: preimage of b}
            Let $G' = Pj(G, \X \sqcup \Y \sqcup \Z)$. If $G'$ contains a subgraph of the form (b), then $G$ contains a subgraph of the form
            \begin{center}
                \scalebox{0.51}{
                % https://q.uiver.app/#q=WzAsMzYsWzAsMCwiWD1VXzAiXSxbMSwxLCJTIl0sWzIsMiwiUyJdLFs2LDAsIlVfMSJdLFs3LDEsIlMiXSxbOCwyLCJTIl0sWzksMywiVCJdLFs5LDQsIlQiXSxbOSw1LCJUIl0sWzksNiwiWl8yIl0sWzEwLDIsIlMiXSxbMTEsMSwiUyJdLFsxMiwwLCJVXzIiXSxbMTMsMSwiUyJdLFsxNCwyLCJTIl0sWzE1LDNdLFsxNiwzLCJcXGNkb3RzIl0sWzE4LDIsIlMiXSxbMTcsM10sWzE5LDEsIlMiXSxbMjAsMCwiVV97ay0xfSJdLFsyMSwxLCJTIl0sWzIyLDIsIlMiXSxbMjMsMywiVCJdLFsyMyw0LCJUIl0sWzIzLDUsIlQiXSxbMjMsNiwiWl9rIl0sWzI0LDIsIlMiXSxbMjUsMSwiUyJdLFsyNiwwLCJZID0gVV9rIl0sWzMsNiwiWl8xIl0sWzMsNSwiVCJdLFszLDQsIlQiXSxbNSwxLCJTIl0sWzQsMiwiUyJdLFszLDMsIlQiXSxbMCwxXSxbMSwyLCJcXGRkb3RzIiwxXSxbMyw0XSxbNCw1LCJcXGRkb3RzIiwxXSxbNSw2XSxbNiw3XSxbNyw4LCJcXHZkb3RzIiwxXSxbOCw5XSxbMTIsMTFdLFsxMiwxM10sWzEzLDE0LCJcXGRkb3RzIiwxXSxbMTQsMTVdLFsxNywxOF0sWzIwLDE5XSxbMjAsMjFdLFsyMSwyMiwiXFxkZG90cyIsMV0sWzIyLDIzXSxbMTksMTcsIlxcaWRkb3RzIiwxXSxbMjMsMjRdLFsyNCwyNSwiXFx2ZG90cyIsMV0sWzI1LDI2XSxbMjcsMjNdLFsyOCwyNywiXFxpZGRvdHMiLDFdLFszMSwzMF0sWzMyLDMxLCJcXHZkb3RzIiwxXSxbMywzM10sWzMzLDM0LCJcXGlkZG90cyIsMV0sWzIsMzVdLFszNSwzMl0sWzM0LDM1XSxbMTEsMTAsIlxcaWRkb3RzIiwxXSxbMTAsNl0sWzI5LDI4XV0=
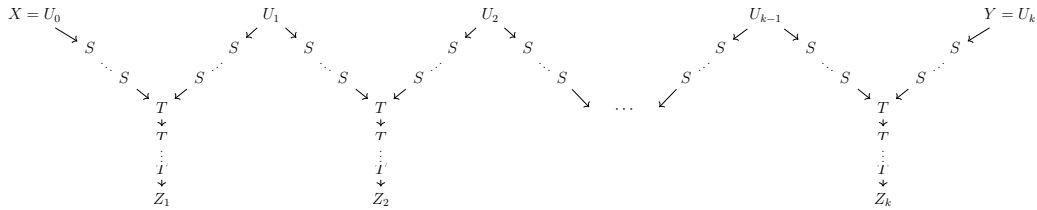
                \begin{tikzcd}[ampersand replacement=\&,sep=tiny]
                    {X=U_0} \&\&\&\&\&\& {U_1} \&\&\&\&\&\& {U_2} \&\&\&\&\&\&\&\& {U_{k-1}} \&\&\&\&\&\& {Y = U_k} \\
                    \& S \&\&\&\& S \&\& S \&\&\&\& S \&\& S \&\&\&\&\&\& S \&\& S \&\&\&\& S \\
                    \&\& S \&\& S \&\&\&\& S \&\& S \&\&\&\& S \&\&\&\& S \&\&\&\& S \&\& S \\
                    \&\&\& T \&\&\&\&\&\& T \&\&\&\&\&\& {} \& \cdots \& {} \&\&\&\&\&\& T \\
                    \&\&\& T \&\&\&\&\&\& T \&\&\&\&\&\&\&\&\&\&\&\&\&\& T \\
                    \&\&\& T \&\&\&\&\&\& T \&\&\&\&\&\&\&\&\&\&\&\&\&\& T \\
                    \&\&\& {Z_1} \&\&\&\&\&\& {Z_2} \&\&\&\&\&\&\&\&\&\&\&\&\&\& {Z_k}
                    \arrow[from=1-1, to=2-2]
                    \arrow[from=1-7, to=2-6]
                    \arrow[from=1-7, to=2-8]
                    \arrow[from=1-13, to=2-12]
                    \arrow[from=1-13, to=2-14]
                    \arrow[from=1-21, to=2-20]
                    \arrow[from=1-21, to=2-22]
                    \arrow[from=1-27, to=2-26]
                    \arrow["\ddots"{description}, from=2-2, to=3-3]
                    \arrow["\iddots"{description}, from=2-6, to=3-5]
                    \arrow["\ddots"{description}, from=2-8, to=3-9]
                    \arrow["\iddots"{description}, from=2-12, to=3-11]
                    \arrow["\ddots"{description}, from=2-14, to=3-15]
                    \arrow["\iddots"{description}, from=2-20, to=3-19]
                    \arrow["\ddots"{description}, from=2-22, to=3-23]
                    \arrow["\iddots"{description}, from=2-26, to=3-25]
                    \arrow[from=3-3, to=4-4]
                    \arrow[from=3-5, to=4-4]
                    \arrow[from=3-9, to=4-10]
                    \arrow[from=3-11, to=4-10]
                    \arrow[from=3-15, to=4-16]
                    \arrow[from=3-19, to=4-18]
                    \arrow[from=3-23, to=4-24]
                    \arrow[from=3-25, to=4-24]
                    \arrow[from=4-4, to=5-4]
                    \arrow[from=4-10, to=5-10]
                    \arrow[from=4-24, to=5-24]
                    \arrow["\vdots"{description}, from=5-4, to=6-4]
                    \arrow["\vdots"{description}, from=5-10, to=6-10]
                    \arrow["\vdots"{description}, from=5-24, to=6-24]
                    \arrow[from=6-4, to=7-4]
                    \arrow[from=6-10, to=7-10]
                    \arrow[from=6-24, to=7-24]
                \end{tikzcd}
                }
                \end{center}
                where $\U, \SSS, \TT \not \subseteq \X,\Y,\Z$.
            \end{lemma}
            \begin{proof}
                This follows directly from the construction of latent projections (Definition~\ref{def: latent_projection}).
            \end{proof}
            The most natural distribution on this pre-image is for the variables $\SSS$ and $\TT$ to copy their parents deterministically. However, this distribution is not positive. Non-positive distributions can, in fact, satisfy identifiability criteria while still violating the corresponding identifying formula.
            
            The solution is to construct a distribution $P$ whose marginal on $(X,Y,\Z)$ is arbitrarily close to that of $P'$. Failure for $P'$ to satisfy the front-door functional then implies failure for $P$ as well. Concretely, we make $\SSS$ and $\TT$ copy their parents with probability $\frac{n}{n+1}$ for $n$ sufficiently large.
        \vspace{6.8pt}

        \noindent \textbf{Step 7: Constructing $P$ on pre-image of (b).} 
        Let $U_0, ..., U_k$ be i.i.d fair coins $\mathcal{U}({0,1})$, with $X = U_0$ and $Y = U_k$. For each $i = 1, ..., k$ define
        $$Z_i \sim \text{ Uniform on the multiset } \left[0,1, Pa(Z_i)\right]$$
        exactly as in the counterexample model of Step 2.

        $\SSS$ and $\TT$ denote the auxiliary variables appearing in the pre-image of~(b), shown in Lemma~\ref{lem: preimage of b}. Let $n \in \mathbb{N}$. For any $S \in \SSS$, let
        \begin{align*}
            S &= \begin{cases} 
            Pa(S) & \text{with probability } \frac{n}{n+1} \\
            \neg Pa(S) & \text{with probability } \frac{1}{n+1} \\
            \end{cases}
        \end{align*}
        For any $T \in \TT$, let $T = (T_\ell, T_r).$ If $S \rightarrow T \leftarrow S'$ in $G$, define
        \begin{align*}
            Pa_\ell(T) := S, \\
            Pa_r(T) := S'.
        \end{align*}
        Otherwise, $T' \rightarrow T$ in $G$, and define
        \begin{align*}
            Pa_\ell(T) := T'_\ell, \quad Pa_r(T) := T'_r.
        \end{align*}
        Then, let
        \begin{align*}
            T_\ell &= \begin{cases} 
            Pa_\ell(T) & \text{with probability } \frac{n}{n+1} \\
            \neg Pa_\ell(T) & \text{with probability } \frac{1}{n+1} \\
            \end{cases}, \quad T_r &= \begin{cases} 
            Pa_r(T) & \text{with probability } \frac{n}{n+1} \\
            \neg Pa_r(T) & \text{with probability } \frac{1}{n+1} \\
            \end{cases} 
        \end{align*}
        \vspace{6.8pt}

        \noindent \textbf{Step 8: $P \to P'$ as $n \to \infty$.}
        By construction, each $S \in \SSS$ copies its parent with probability $\frac{n}{n+1}$, and each $T \in \TT$ copies its parent(s) with probability $\left(\frac{n}{n+1}\right)^2$. For sufficiently large $n$, this copying is near-perfect.

        \begin{proposition} \label{prop: P_n(E_n) -> 1}
            Fix a pre-image of (b) in $G$, and let $E$ denote the event that all $S \in \SSS$ and all components of $T \in \TT$ exactly equal their assigned parents. Then
            \[
            \lim_{n\to\infty} P_n(E) = 1.
            \]
        \end{proposition}
        \begin{proof}
            $$P_n(E) = \left(\frac{n}{n+1}\right)^{|\SSS| + 2|\TT|},$$
            where $|\SSS| + 2|\TT|$ is fixed. Taking the limit gives
            $$\lim_{n\to\infty}P_n(E)= 1.$$
            % For a single $S$-node with edge $Pa(S) \rightarrow S$, the conditional probability of perfect copying is
            % $$P_n(S = Pa(S) \mid Pa(S)) = \frac{n+1}{n+2}.$$
            % Since the graph is finite and each variable $S \in \SSS$ is independent conditional on their parents, the probability of the joint event $E_n$ is the product over all $S \in \SSS$,
            % $$P_n(E_n) = \left(\frac{n+1}{n+2}\right)^M.$$
            % For any fixed $M$,
            % $$\lim_{n\to\infty}P(E_n)= 1.$$
        \end{proof}
        Conditioned on perfect copying, $P$ is identical to $P'$ on $(X,Y,\Z)$:
        \begin{proposition} \label{lem: P_n | E_n = P'}
            $$P_n(X, Y, \Z \mid E) = P'(X,Y,\Z).$$
        \end{proposition}
        \begin{proof}
            The source nodes $U_0, \dots, U_k$ are i.i.d. fair coins in both constructions, and $X = U_0$, $Y = U_k$. In the event of perfect copying, for each $i = 1, \ldots, k$, $$Pa(Z_i) = \{U_{i-1}, U_i\}$$ as values, the same as in $P'$.  The source nodes both in (b) and its preimage are i.i.d. fair coins. Consequently, the joint distribution of $(X,Y,\Z)$ under $P$ conditioned on $E$ coincides exactly with $P'$,
            \[
            P_n(X, Y, \Z \mid E) = P'(X,Y,\Z).
            \]
        \end{proof}
    
        \begin{lemma} \label{lem: P -> P'}
            $$\lim_{n\to\infty}P_n(x, y, \z) = P'(x, y, \z).$$
        \end{lemma}

        \begin{proof}
            Expand by law of total probability
            \begin{align*}
                P_n(x, y, \z) &= P_n(x, y, \z \mid E)P(E) + P_n(x, y, \z \mid E^c) P_n(E^c) \\
                &= P'(x, y,\z) P_n(E) + P_n(x, y, \z \mid E^c) P_n(E^c) \tag{By Proposition~\ref{lem: P_n | E_n = P'}}
            \end{align*}
            As $n \to \infty$, $P_n(E) = 1$ and $P_n(E^c) = 0$ by Proposition~\ref{prop: P_n(E_n) -> 1}. So,
            $$\lim_{n\to\infty}P_n(x, y, \z) = P'(x, y, \z).$$
        \end{proof}
        
        Thus, for sufficiently large $n$, the distribution $P_n$ on the pre-image of (b) is arbitrarily close to $P'$, and the front-door formula fails. One can extend $P_n$ to a distribution $P$ on the full graph $G$ by Lemma~\ref{lem: fdf_fails_in_subgraph_implies_fails_in_supgraph}; This yields the desired counterexample.
        % \vspace{6.8pt}
        % \noindent \textbf{Step 9: Front-door formula fails for $n$ sufficiently large.}
        % Recall the definition of $S$ as the front-door formula evaluated at $x^* = 0, y = 0$ for $P'$. We define $S$ analogously for $P$. $S$ is a function of $P(x, y, \z)$. By \Cref{lem: P -> P'}, $$\lim_{n\to\infty}S = S.$$

        % The causal effect $P_{n,X}(y) = P_n(y)$ by Rule 3 of do-calc, so by \Cref{lem: P -> P'}, $$\lim_{n\to\infty}P_{n, X}(y) = P'(y) = P'_{X}(y).$$

        % Since the front-door formula fails for $P'$ ($S \neq P'_{X}(y))$), therefore
        % $$\lim_{n \to \infty}(S_n - P_{n,X}(y)) = S - P'_X(y) \neq 0.$$
        % This implies for sufficiently large $n$, $S_n \neq P_{n,X}(y)$. Hence we can choose sufficiently large $n$ such that the front-door formula fails for $P := P_n$ on the subgraph of $G$, and therefore fails on $G$ overall.
        \vspace{6.8pt}

        \noindent \textbf{Step 10: Pre-image of (c) in G and construction of P.}
        The process is near-identical to Steps 7-9 for (b). The graph $G$ will contain a subgraph of the form below.
        \begin{center}
            \scalebox{0.51}{
            % https://q.uiver.app/#q=WzAsMzksWzAsMCwiWD1VXzAiXSxbMSwxLCJTIl0sWzIsMiwiUyJdLFs2LDAsIlVfMSJdLFs3LDEsIlMiXSxbOCwyLCJTIl0sWzksMywiVCJdLFs5LDQsIlQiXSxbOSw1LCJUIl0sWzksNiwiWl8yIl0sWzEwLDIsIlMiXSxbMTEsMSwiUyJdLFsxMiwwLCJVXzIiXSxbMTMsMSwiUyJdLFsxNCwyLCJTIl0sWzE1LDNdLFsxNiwzLCJcXGNkb3RzIl0sWzE4LDIsIlMiXSxbMTcsM10sWzE5LDEsIlMiXSxbMjAsMCwiVV97ay0xfSJdLFsyMSwxLCJTIl0sWzIyLDIsIlMiXSxbMjMsMywiVCJdLFsyMyw0LCJUIl0sWzIzLDUsIlQiXSxbMjMsNiwiWl9rIl0sWzI0LDIsIlMiXSxbMjUsMSwiUyJdLFsyNiwwLCJVX2siXSxbMyw2LCJaXzEiXSxbMyw1LCJUIl0sWzMsNCwiVCJdLFs1LDEsIlMiXSxbNCwyLCJTIl0sWzMsMywiVCJdLFsyNywxLCJTIl0sWzI4LDIsIlMiXSxbMjksMywiWSJdLFswLDFdLFsxLDIsIlxcZGRvdHMiLDFdLFszLDRdLFs0LDUsIlxcZGRvdHMiLDFdLFs1LDZdLFs2LDddLFs3LDgsIlxcdmRvdHMiLDFdLFs4LDldLFsxMiwxMV0sWzEyLDEzXSxbMTMsMTQsIlxcZGRvdHMiLDFdLFsxNCwxNV0sWzE3LDE4XSxbMjAsMTldLFsyMCwyMV0sWzIxLDIyLCJcXGRkb3RzIiwxXSxbMjIsMjNdLFsxOSwxNywiXFxpZGRvdHMiLDFdLFsyMywyNF0sWzI0LDI1LCJcXHZkb3RzIiwxXSxbMjUsMjZdLFsyNywyM10sWzI4LDI3LCJcXGlkZG90cyIsMV0sWzMxLDMwXSxbMzIsMzEsIlxcdmRvdHMiLDFdLFszLDMzXSxbMzMsMzQsIlxcaWRkb3RzIiwxXSxbMiwzNV0sWzM1LDMyXSxbMzQsMzVdLFsyOSwzNl0sWzM2LDM3XSxbMzcsMzhdLFsxMSwxMCwiXFxpZGRvdHMiLDFdLFsxMCw2XSxbMjksMjhdXQ==
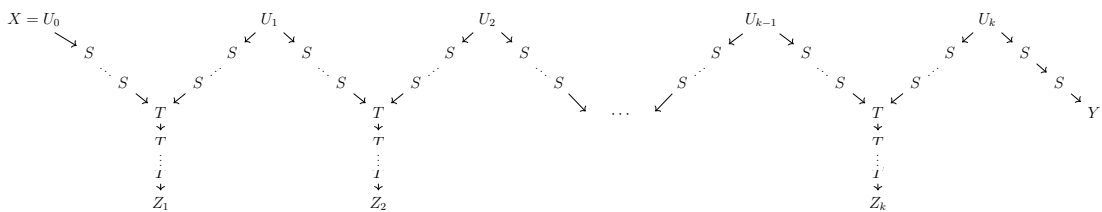
            \begin{tikzcd}[ampersand replacement=\&,sep=tiny]
                {X=U_0} \&\&\&\&\&\& {U_1} \&\&\&\&\&\& {U_2} \&\&\&\&\&\&\&\& {U_{k-1}} \&\&\&\&\&\& {U_k} \\
                \& S \&\&\&\& S \&\& S \&\&\&\& S \&\& S \&\&\&\&\&\& S \&\& S \&\&\&\& S \&\& S \\
                \&\& S \&\& S \&\&\&\& S \&\& S \&\&\&\& S \&\&\&\& S \&\&\&\& S \&\& S \&\&\&\& S \\
                \&\&\& T \&\&\&\&\&\& T \&\&\&\&\&\& {} \& \cdots \& {} \&\&\&\&\&\& T \&\&\&\&\&\& Y \\
                \&\&\& T \&\&\&\&\&\& T \&\&\&\&\&\&\&\&\&\&\&\&\&\& T \\
                \&\&\& T \&\&\&\&\&\& T \&\&\&\&\&\&\&\&\&\&\&\&\&\& T \\
                \&\&\& {Z_1} \&\&\&\&\&\& {Z_2} \&\&\&\&\&\&\&\&\&\&\&\&\&\& {Z_k}
                \arrow[from=1-1, to=2-2]
                \arrow[from=1-7, to=2-6]
                \arrow[from=1-7, to=2-8]
                \arrow[from=1-13, to=2-12]
                \arrow[from=1-13, to=2-14]
                \arrow[from=1-21, to=2-20]
                \arrow[from=1-21, to=2-22]
                \arrow[from=1-27, to=2-26]
                \arrow[from=1-27, to=2-28]
                \arrow["\ddots"{description}, from=2-2, to=3-3]
                \arrow["\iddots"{description}, from=2-6, to=3-5]
                \arrow["\ddots"{description}, from=2-8, to=3-9]
                \arrow["\iddots"{description}, from=2-12, to=3-11]
                \arrow["\ddots"{description}, from=2-14, to=3-15]
                \arrow["\iddots"{description}, from=2-20, to=3-19]
                \arrow["\ddots"{description}, from=2-22, to=3-23]
                \arrow["\iddots"{description}, from=2-26, to=3-25]
                \arrow[from=2-28, to=3-29]
                \arrow[from=3-3, to=4-4]
                \arrow[from=3-5, to=4-4]
                \arrow[from=3-9, to=4-10]
                \arrow[from=3-11, to=4-10]
                \arrow[from=3-15, to=4-16]
                \arrow[from=3-19, to=4-18]
                \arrow[from=3-23, to=4-24]
                \arrow[from=3-25, to=4-24]
                \arrow[from=3-29, to=4-30]
                \arrow[from=4-4, to=5-4]
                \arrow[from=4-10, to=5-10]
                \arrow[from=4-24, to=5-24]
                \arrow["\vdots"{description}, from=5-4, to=6-4]
                \arrow["\vdots"{description}, from=5-10, to=6-10]
                \arrow["\vdots"{description}, from=5-24, to=6-24]
                \arrow[from=6-4, to=7-4]
                \arrow[from=6-10, to=7-10]
                \arrow[from=6-24, to=7-24]
            \end{tikzcd}
            }
        \end{center}
        where $\U, \SSS, \TT \not \subseteq \X,\Y,\Z$.
        
        \noindent \textbf{Conclusion.} If $G$ violates (ii), there exists a probability distribution $P$ on $G$ for which $P_\X(\y)$ is not identified by the front-door formula.
    \end{proof}
\section{Concluding Remarks and Further Questions}
\label{sec: concluding}
    We have presented a new graphical criterion that strictly expands the class of causal graphs for which the front-door functional $P_\X(\y) = \sum_\z P(\z \mid \x) \sum_{\x'} P(\y \mid \z, \x') P(\x')$ is valid.

    However, our new criterion is still incomplete. We have established Condition (ii) ($\Z$ blocks all proper directed paths from $\X$ to $\Y$) is necessary, but Condition (i) (no back-door paths from $\X$ to $\Z$) is not, as shown by the counterexamples in \Cref{fig: fdc_2}(b).

    \begin{center}
        \scalebox{0.51}{% https://q.uiver.app/#q=WzAsMTI0LFs3LDEsIlgiXSxbOCwxLCJaXzEiXSxbOSwxLCJZIl0sWzgsMCwiWl8yIl0sWzUsMCwiWiJdLFs0LDAsIlgiXSxbNiwwLCJZIl0sWzExLDEsIlgiXSxbMTIsMSwiWl8xIl0sWzEzLDEsIlkiXSxbMTMsMCwiWl80Il0sWzExLDAsIlpfMiJdLFsxMiwwLCJaXzMiXSxbMCw3LCJYIl0sWzEsNywiWl8xIl0sWzIsNywiWl8yIl0sWzMsNywiWSJdLFsyLDYsIlpfMyJdLFsxLDksIlgiXSxbMiw5LCJaXzEiXSxbMyw5LCJZIl0sWzIsOCwiWl8yIl0sWzEsMSwiWCJdLFsyLDEsIlpfMSJdLFszLDEsIlkiXSxbMiwwLCJaXzIiXSxbMiwyLCJaXzMiXSxbMCw0LCJYIl0sWzEsNCwiWl8xIl0sWzMsNCwiWSJdLFsxLDUsIlpfNCJdLFsyLDQsIlpfMiJdLFsyLDMsIlpfMyJdLFs0LDIsIlhfMSJdLFs1LDIsIlpfMSJdLFs2LDIsIlkiXSxbNCwxLCJaXzIiXSxbNSwxLCJYIl0sWzYsMSwiWl8zIl0sWzQsNCwiWF8xIl0sWzUsNCwiWl8xIl0sWzYsNCwiWSJdLFs0LDMsIlpfMiJdLFs1LDMsIlgiXSxbNiwzLCJaXzMiXSxbNCw2LCJYXzEiXSxbNSw2LCJaXzEiXSxbNiw2LCJZIl0sWzQsNSwiWl8yIl0sWzUsNSwiWCJdLFs2LDUsIlpfMyJdLFs0LDgsIlhfMSJdLFs1LDgsIlpfMSJdLFs2LDgsIlkiXSxbNCw3LCJaXzIiXSxbNSw3LCJYIl0sWzYsNywiWl8zIl0sWzExLDMsIlhfMSJdLFsxMiwzLCJaXzEiXSxbMTMsMywiWSJdLFsxMSwyLCJaXzIiXSxbMTIsMiwiWCJdLFsxMywyLCJaXzMiXSxbMTEsNSwiWF8xIl0sWzEyLDUsIlpfMSJdLFsxMyw1LCJZIl0sWzExLDQsIlpfMiJdLFsxMiw0LCJYIl0sWzEzLDQsIlpfMyJdLFsxNSw5LCJYXzIiXSxbMTQsOSwiWF8xIl0sWzE2LDksIloiXSxbMTcsOSwiWSJdLFsxMSw2LCJYIl0sWzEyLDYsIlpfMSJdLFsxMyw2LCJZIl0sWzEyLDcsIlpfMiJdLFsxMSw4LCJYIl0sWzEyLDgsIlpfMSJdLFsxMyw4LCJZIl0sWzEyLDksIlpfMiJdLFs3LDYsIlgiXSxbOCw2LCJaXzEiXSxbOSw2LCJZIl0sWzgsNywiWl8yIl0sWzcsOCwiWCJdLFs4LDgsIlpfMSJdLFs5LDgsIlkiXSxbOCw5LCJaXzIiXSxbMTQsMCwiWCJdLFsxNSwwLCJaXzEiXSxbMTYsMCwiWSJdLFsxNCwxLCJaXzIiXSxbMTUsMSwiWl8zIl0sWzE0LDIsIlgiXSxbMTUsMiwiWl8xIl0sWzE2LDIsIlkiXSxbMTQsMywiWl8yIl0sWzE1LDMsIlpfMyJdLFs3LDIsIlgiXSxbOCwyLCJaXzEiXSxbOSwyLCJZIl0sWzgsMywiWl8zIl0sWzE1LDQsIlhfMiJdLFsxNCw0LCJaXzQiXSxbMTUsNSwiWF8xIl0sWzE2LDUsIlpfMSJdLFsxNyw1LCJZIl0sWzE2LDQsIlpfMyJdLFsxNiw2LCJaXzIiXSxbMTQsOCwiWCJdLFsxNSw4LCJaXzEiXSxbMTYsOCwiWSJdLFsxNCw3LCJaXzIiXSxbMTUsNywiWl8zIl0sWzcsNSwiWCJdLFs4LDUsIlpfMSJdLFs5LDUsIlkiXSxbNyw0LCJaXzIiXSxbOCw0LCJaXzMiXSxbMTAsMF0sWzEwLDEwXSxbNSwxMCwiXFx0ZXh0eyhhKX0iXSxbMTQsMTAsIlxcdGV4dHsoYil9Il0sWzAsMV0sWzEsMl0sWzMsMiwiIiwyLHsic3R5bGUiOnsidGFpbCI6eyJuYW1lIjoiYXJyb3doZWFkIn0sImJvZHkiOnsibmFtZSI6ImRhc2hlZCJ9fX1dLFszLDBdLFs0LDVdLFs0LDYsIiIsMix7InN0eWxlIjp7InRhaWwiOnsibmFtZSI6ImFycm93aGVhZCJ9LCJib2R5Ijp7Im5hbWUiOiJkYXNoZWQifX19XSxbNyw4XSxbOCw5XSxbOSwxMF0sWzcsMTFdLFsxMSw4XSxbMTIsMTFdLFsxMiwxMCwiIiwxLHsic3R5bGUiOnsidGFpbCI6eyJuYW1lIjoiYXJyb3doZWFkIn0sImJvZHkiOnsibmFtZSI6ImRhc2hlZCJ9fX1dLFs3LDksIiIsMSx7ImN1cnZlIjozLCJzdHlsZSI6eyJ0YWlsIjp7Im5hbWUiOiJhcnJvd2hlYWQifSwiYm9keSI6eyJuYW1lIjoiZGFzaGVkIn19fV0sWzE1LDE3XSxbMTYsMTcsIiIsMSx7InN0eWxlIjp7InRhaWwiOnsibmFtZSI6ImFycm93aGVhZCJ9LCJib2R5Ijp7Im5hbWUiOiJkYXNoZWQifX19XSxbMTMsMTRdLFsxNCwxNV0sWzE1LDE2XSxbMTMsMTUsIiIsMSx7ImN1cnZlIjoyLCJzdHlsZSI6eyJ0YWlsIjp7Im5hbWUiOiJhcnJvd2hlYWQifSwiYm9keSI6eyJuYW1lIjoiZGFzaGVkIn19fV0sWzEzLDE2LCIiLDEseyJjdXJ2ZSI6Mywic3R5bGUiOnsidGFpbCI6eyJuYW1lIjoiYXJyb3doZWFkIn0sImJvZHkiOnsibmFtZSI6ImRhc2hlZCJ9fX1dLFsxOSwxOF0sWzIwLDE5LCIiLDAseyJzdHlsZSI6eyJ0YWlsIjp7Im5hbWUiOiJhcnJvd2hlYWQifSwiYm9keSI6eyJuYW1lIjoiZGFzaGVkIn19fV0sWzIxLDIwXSxbMjMsMjJdLFsyNCwyMywiIiwwLHsic3R5bGUiOnsidGFpbCI6eyJuYW1lIjoiYXJyb3doZWFkIn0sImJvZHkiOnsibmFtZSI6ImRhc2hlZCJ9fX1dLFsyNSwyMywiIiwwLHsic3R5bGUiOnsidGFpbCI6eyJuYW1lIjoiYXJyb3doZWFkIn0sImJvZHkiOnsibmFtZSI6ImRhc2hlZCJ9fX1dLFsyMiwyNl0sWzI2LDI0XSxbMjgsMjddLFsyNywzMF0sWzMwLDI5XSxbMzEsMjhdLFsyOSwzMSwiIiwwLHsic3R5bGUiOnsidGFpbCI6eyJuYW1lIjoiYXJyb3doZWFkIn0sImJvZHkiOnsibmFtZSI6ImRhc2hlZCJ9fX1dLFszMiwzMSwiIiwwLHsic3R5bGUiOnsidGFpbCI6eyJuYW1lIjoiYXJyb3doZWFkIn0sImJvZHkiOnsibmFtZSI6ImRhc2hlZCJ9fX1dLFszMywzNF0sWzM0LDM1XSxbMzYsMzMsIiIsMCx7InN0eWxlIjp7InRhaWwiOnsibmFtZSI6ImFycm93aGVhZCJ9LCJib2R5Ijp7Im5hbWUiOiJkYXNoZWQifX19XSxbMzYsMzddLFszOCwzN10sWzM4LDM1XSxbMzksNDBdLFs0MCw0MV0sWzQyLDM5LCIiLDAseyJzdHlsZSI6eyJ0YWlsIjp7Im5hbWUiOiJhcnJvd2hlYWQifSwiYm9keSI6eyJuYW1lIjoiZGFzaGVkIn19fV0sWzQ0LDQzXSxbNDQsNDFdLFsyMSwxOV0sWzQ1LDQ2XSxbNDYsNDddLFs0OCw0NSwiIiwwLHsic3R5bGUiOnsidGFpbCI6eyJuYW1lIjoiYXJyb3doZWFkIn0sImJvZHkiOnsibmFtZSI6ImRhc2hlZCJ9fX1dLFs0OCw0OV0sWzUwLDQ5XSxbNTEsNTJdLFs1Miw1M10sWzU0LDUxLCIiLDAseyJzdHlsZSI6eyJ0YWlsIjp7Im5hbWUiOiJhcnJvd2hlYWQifSwiYm9keSI6eyJuYW1lIjoiZGFzaGVkIn19fV0sWzU2LDU1XSxbNDMsNDJdLFs0Nyw1MF0sWzUzLDU2XSxbNTUsNTRdLFs1Nyw1OF0sWzU4LDU5XSxbNjAsNTcsIiIsMCx7InN0eWxlIjp7InRhaWwiOnsibmFtZSI6ImFycm93aGVhZCJ9LCJib2R5Ijp7Im5hbWUiOiJkYXNoZWQifX19XSxbNjAsNjFdLFs2Miw1OV0sWzYzLDY0XSxbNjQsNjVdLFs2Niw2MywiIiwwLHsic3R5bGUiOnsidGFpbCI6eyJuYW1lIjoiYXJyb3doZWFkIn0sImJvZHkiOnsibmFtZSI6ImRhc2hlZCJ9fX1dLFs2OCw2NV0sWzYxLDYyXSxbNjcsNjhdLFs2Nyw2Nl0sWzY5LDcxXSxbNzEsNzJdLFs2OSw3MF0sWzczLDc0XSxbNzQsNzVdLFs3Niw3MywiIiwxLHsic3R5bGUiOnsidGFpbCI6eyJuYW1lIjoiYXJyb3doZWFkIn0sImJvZHkiOnsibmFtZSI6ImRhc2hlZCJ9fX1dLFs3Nyw3OF0sWzc4LDc5XSxbODAsNzddLFs4Miw4M10sWzg0LDgxXSxbODIsODFdLFs4OCw4NV0sWzg2LDg1XSxbODcsODZdLFs4OSw5MF0sWzkwLDkxXSxbODksOTIsIiIsMCx7InN0eWxlIjp7InRhaWwiOnsibmFtZSI6ImFycm93aGVhZCJ9LCJib2R5Ijp7Im5hbWUiOiJkYXNoZWQifX19XSxbOTMsOTEsIiIsMCx7InN0eWxlIjp7InRhaWwiOnsibmFtZSI6ImFycm93aGVhZCJ9LCJib2R5Ijp7Im5hbWUiOiJkYXNoZWQifX19XSxbOTQsOTVdLFs5NSw5Nl0sWzk0LDk3LCIiLDAseyJzdHlsZSI6eyJ0YWlsIjp7Im5hbWUiOiJhcnJvd2hlYWQifSwiYm9keSI6eyJuYW1lIjoiZGFzaGVkIn19fV0sWzk3LDk4XSxbOTYsOThdLFs5OSwxMDBdLFsxMDAsMTAxXSxbMTAyLDEwMV0sWzk5LDEwMiwiIiwyLHsic3R5bGUiOnsidGFpbCI6eyJuYW1lIjoiYXJyb3doZWFkIn0sImJvZHkiOnsibmFtZSI6ImRhc2hlZCJ9fX1dLFsxMDQsMTAzLCIiLDIseyJzdHlsZSI6eyJ0YWlsIjp7Im5hbWUiOiJhcnJvd2hlYWQifSwiYm9keSI6eyJuYW1lIjoiZGFzaGVkIn19fV0sWzEwNSwxMDZdLFsxMDYsMTA3XSxbMTAzLDEwOF0sWzEwOCwxMDddLFsxMDMsMTA1XSxbMTA2LDEwOV0sWzEwNywxMDksIiIsMSx7InN0eWxlIjp7InRhaWwiOnsibmFtZSI6ImFycm93aGVhZCJ9LCJib2R5Ijp7Im5hbWUiOiJkYXNoZWQifX19XSxbMTA3LDEwNSwiIiwxLHsiY3VydmUiOjIsInN0eWxlIjp7InRhaWwiOnsibmFtZSI6ImFycm93aGVhZCJ9LCJib2R5Ijp7Im5hbWUiOiJkYXNoZWQifX19XSxbMTEwLDExMV0sWzExMSwxMTJdLFsxMTMsMTEwXSxbMTEzLDExNF0sWzExMSwxMTRdLFsxMTQsMTEyLCIiLDEseyJzdHlsZSI6eyJ0YWlsIjp7Im5hbWUiOiJhcnJvd2hlYWQifSwiYm9keSI6eyJuYW1lIjoiZGFzaGVkIn19fV0sWzExNSwxMTZdLFsxMTYsMTE3XSxbMTE4LDExNV0sWzExOCwxMTldLFsxMTYsMTE5XSxbMTE5LDExNywiIiwxLHsic3R5bGUiOnsidGFpbCI6eyJuYW1lIjoiYXJyb3doZWFkIn0sImJvZHkiOnsibmFtZSI6ImRhc2hlZCJ9fX1dLFsxMTUsMTE3LCIiLDEseyJjdXJ2ZSI6Mywic3R5bGUiOnsidGFpbCI6eyJuYW1lIjoiYXJyb3doZWFkIn0sImJvZHkiOnsibmFtZSI6ImRhc2hlZCJ9fX1dLFs5Miw5M10sWzEyMCwxMjEsIiIsMSx7InN0eWxlIjp7ImhlYWQiOnsibmFtZSI6Im5vbmUifX19XV0=
        \begin{tikzcd}[ampersand replacement=\&]
        	\&\& {Z_2} \&\& X \& Z \& Y \&\& {Z_2} \&\& {} \& {Z_2} \& {Z_3} \& {Z_4} \& X \& {Z_1} \& Y \\
        	\& X \& {Z_1} \& Y \& {Z_2} \& X \& {Z_3} \& X \& {Z_1} \& Y \&\& X \& {Z_1} \& Y \& {Z_2} \& {Z_3} \\
        	\&\& {Z_3} \&\& {X_1} \& {Z_1} \& Y \& X \& {Z_1} \& Y \&\& {Z_2} \& X \& {Z_3} \& X \& {Z_1} \& Y \\
        	\&\& {Z_3} \&\& {Z_2} \& X \& {Z_3} \&\& {Z_3} \&\&\& {X_1} \& {Z_1} \& Y \& {Z_2} \& {Z_3} \\
        	X \& {Z_1} \& {Z_2} \& Y \& {X_1} \& {Z_1} \& Y \& {Z_2} \& {Z_3} \&\&\& {Z_2} \& X \& {Z_3} \& {Z_4} \& {X_2} \& {Z_3} \\
        	\& {Z_4} \&\&\& {Z_2} \& X \& {Z_3} \& X \& {Z_1} \& Y \&\& {X_1} \& {Z_1} \& Y \&\& {X_1} \& {Z_1} \& Y \\
        	\&\& {Z_3} \&\& {X_1} \& {Z_1} \& Y \& X \& {Z_1} \& Y \&\& X \& {Z_1} \& Y \&\&\& {Z_2} \\
        	X \& {Z_1} \& {Z_2} \& Y \& {Z_2} \& X \& {Z_3} \&\& {Z_2} \&\&\&\& {Z_2} \&\& {Z_2} \& {Z_3} \\
        	\&\& {Z_2} \&\& {X_1} \& {Z_1} \& Y \& X \& {Z_1} \& Y \&\& X \& {Z_1} \& Y \& X \& {Z_1} \& Y \\
        	\& X \& {Z_1} \& Y \&\&\&\&\& {Z_2} \&\&\&\& {Z_2} \&\& {X_1} \& {X_2} \& Z \& Y \\
        	\&\&\&\&\& {\text{(a)}} \&\&\&\&\& {} \&\&\&\& {\text{(b)}}
        	\arrow[dashed, tail reversed, from=1-3, to=2-3]
        	\arrow[from=1-6, to=1-5]
        	\arrow[dashed, tail reversed, from=1-6, to=1-7]
        	\arrow[from=1-9, to=2-8]
        	\arrow[dashed, tail reversed, from=1-9, to=2-10]
        	\arrow[no head, from=1-11, to=11-11]
        	\arrow[from=1-12, to=2-13]
        	\arrow[from=1-13, to=1-12]
        	\arrow[dashed, tail reversed, from=1-13, to=1-14]
        	\arrow[from=1-15, to=1-16]
        	\arrow[dashed, tail reversed, from=1-15, to=2-15]
        	\arrow[from=1-16, to=1-17]
        	\arrow[from=2-2, to=3-3]
        	\arrow[from=2-3, to=2-2]
        	\arrow[dashed, tail reversed, from=2-4, to=2-3]
        	\arrow[from=2-5, to=2-6]
        	\arrow[dashed, tail reversed, from=2-5, to=3-5]
        	\arrow[from=2-7, to=2-6]
        	\arrow[from=2-7, to=3-7]
        	\arrow[from=2-8, to=2-9]
        	\arrow[from=2-9, to=2-10]
        	\arrow[from=2-12, to=1-12]
        	\arrow[from=2-12, to=2-13]
        	\arrow[curve={height=18pt}, dashed, tail reversed, from=2-12, to=2-14]
        	\arrow[from=2-13, to=2-14]
        	\arrow[from=2-14, to=1-14]
        	\arrow[from=2-15, to=2-16]
        	\arrow[dashed, tail reversed, from=2-16, to=1-17]
        	\arrow[from=3-3, to=2-4]
        	\arrow[from=3-5, to=3-6]
        	\arrow[from=3-6, to=3-7]
        	\arrow[from=3-8, to=3-9]
        	\arrow[dashed, tail reversed, from=3-8, to=4-9]
        	\arrow[from=3-9, to=3-10]
        	\arrow[from=3-12, to=3-13]
        	\arrow[dashed, tail reversed, from=3-12, to=4-12]
        	\arrow[from=3-13, to=3-14]
        	\arrow[from=3-14, to=4-14]
        	\arrow[from=3-15, to=3-16]
        	\arrow[dashed, tail reversed, from=3-15, to=4-15]
        	\arrow[from=3-16, to=3-17]
        	\arrow[from=3-17, to=4-16]
        	\arrow[dashed, tail reversed, from=4-3, to=5-3]
        	\arrow[dashed, tail reversed, from=4-5, to=5-5]
        	\arrow[from=4-6, to=4-5]
        	\arrow[from=4-7, to=4-6]
        	\arrow[from=4-7, to=5-7]
        	\arrow[from=4-9, to=3-10]
        	\arrow[from=4-12, to=4-13]
        	\arrow[from=4-13, to=4-14]
        	\arrow[from=4-15, to=4-16]
        	\arrow[from=5-1, to=6-2]
        	\arrow[from=5-2, to=5-1]
        	\arrow[from=5-3, to=5-2]
        	\arrow[dashed, tail reversed, from=5-4, to=5-3]
        	\arrow[from=5-5, to=5-6]
        	\arrow[from=5-6, to=5-7]
        	\arrow[from=5-8, to=5-9]
        	\arrow[from=5-8, to=6-8]
        	\arrow[dashed, tail reversed, from=5-9, to=6-10]
        	\arrow[dashed, tail reversed, from=5-12, to=6-12]
        	\arrow[from=5-13, to=5-12]
        	\arrow[from=5-13, to=5-14]
        	\arrow[from=5-14, to=6-14]
        	\arrow[dashed, tail reversed, from=5-15, to=5-16]
        	\arrow[from=5-16, to=5-17]
        	\arrow[from=5-16, to=6-16]
        	\arrow[from=5-17, to=6-18]
        	\arrow[from=6-2, to=5-4]
        	\arrow[from=6-5, to=6-6]
        	\arrow[dashed, tail reversed, from=6-5, to=7-5]
        	\arrow[from=6-7, to=6-6]
        	\arrow[from=6-8, to=6-9]
        	\arrow[curve={height=18pt}, dashed, tail reversed, from=6-8, to=6-10]
        	\arrow[from=6-9, to=5-9]
        	\arrow[from=6-9, to=6-10]
        	\arrow[from=6-12, to=6-13]
        	\arrow[from=6-13, to=6-14]
        	\arrow[from=6-16, to=6-17]
        	\arrow[from=6-17, to=6-18]
        	\arrow[from=6-17, to=7-17]
        	\arrow[curve={height=12pt}, dashed, tail reversed, from=6-18, to=6-16]
        	\arrow[dashed, tail reversed, from=6-18, to=7-17]
        	\arrow[from=7-5, to=7-6]
        	\arrow[from=7-6, to=7-7]
        	\arrow[from=7-7, to=6-7]
        	\arrow[from=7-9, to=7-8]
        	\arrow[from=7-9, to=7-10]
        	\arrow[from=7-12, to=7-13]
        	\arrow[from=7-13, to=7-14]
        	\arrow[from=8-1, to=8-2]
        	\arrow[curve={height=12pt}, dashed, tail reversed, from=8-1, to=8-3]
        	\arrow[curve={height=18pt}, dashed, tail reversed, from=8-1, to=8-4]
        	\arrow[from=8-2, to=8-3]
        	\arrow[from=8-3, to=7-3]
        	\arrow[from=8-3, to=8-4]
        	\arrow[dashed, tail reversed, from=8-4, to=7-3]
        	\arrow[dashed, tail reversed, from=8-5, to=9-5]
        	\arrow[from=8-6, to=8-5]
        	\arrow[from=8-7, to=8-6]
        	\arrow[from=8-9, to=7-8]
        	\arrow[dashed, tail reversed, from=8-13, to=7-12]
        	\arrow[from=8-15, to=8-16]
        	\arrow[from=8-15, to=9-15]
        	\arrow[dashed, tail reversed, from=8-16, to=9-17]
        	\arrow[from=9-3, to=10-3]
        	\arrow[from=9-3, to=10-4]
        	\arrow[from=9-5, to=9-6]
        	\arrow[from=9-6, to=9-7]
        	\arrow[from=9-7, to=8-7]
        	\arrow[from=9-9, to=9-8]
        	\arrow[from=9-10, to=9-9]
        	\arrow[from=9-12, to=9-13]
        	\arrow[from=9-13, to=9-14]
        	\arrow[from=9-15, to=9-16]
        	\arrow[from=9-16, to=8-16]
        	\arrow[from=9-16, to=9-17]
        	\arrow[from=10-3, to=10-2]
        	\arrow[dashed, tail reversed, from=10-4, to=10-3]
        	\arrow[from=10-9, to=9-8]
        	\arrow[from=10-13, to=9-12]
        	\arrow[from=10-16, to=10-15]
        	\arrow[from=10-16, to=10-17]
        	\arrow[from=10-17, to=10-18]
        \end{tikzcd}}
    
    \end{center}
    
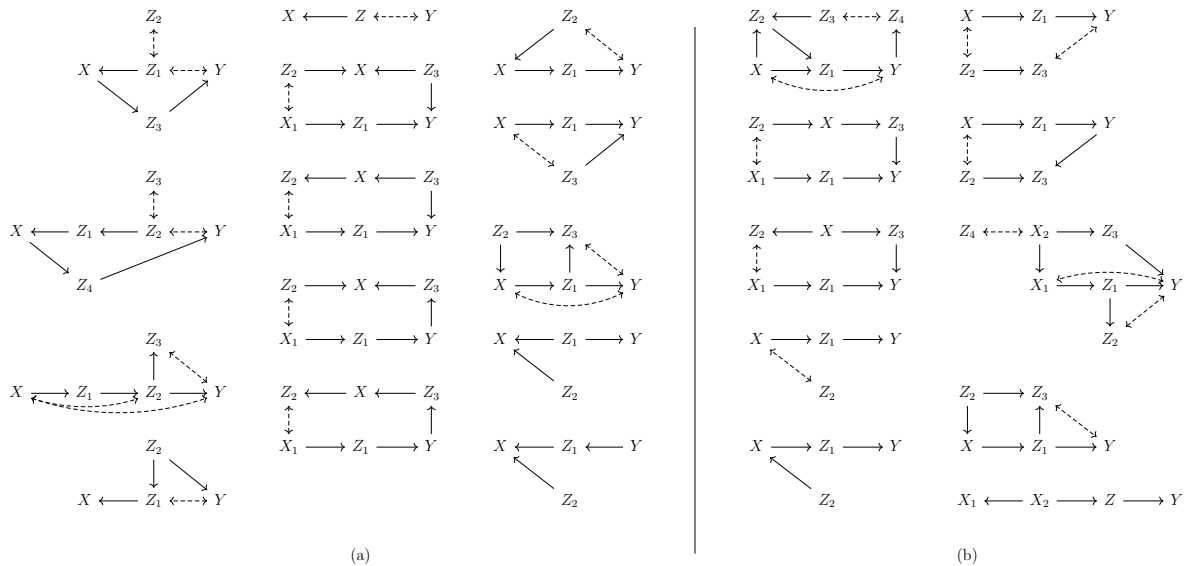
\captionof{figure}{Graphs violating condition (i) where (a) $P_\X(\y)$ does not match the front-door functional; (b) $P_\X(\y)$ \textit{does} match by the front-door functional.} \label{fig: fdc_2}

    The most immediate next step is to resolve this incompleteness by finding a necessary and sufficient graphical criterion for the front-door functional. Future work could pursue this goal or seek to characterize other common identifying functionals.

\bibliographystyle{plain}
\bibliography{sources}    

\appendix

\section{Intro to Causal Inference} \label{apx: prelim prelim}.
    \subsection{Basic Definitions}
        \begin{definition}[Directed Acyclic Graph (DAG)]
            A \textit{directed acyclic graph (DAG)} is a graph $G = (\V, \E)$ where each edge $E \in \E$ is directed (i.e., of the form $V_1 \to V_2$ where $V_1, V_2 \in \V$), and $G$ contains no directed cycles.
        \end{definition}

        We denote parents, ancestors, children, and descendants of a set of nodes $\Z \subseteq G$ as $Pa(\Z)$, $An(\Z)$, $Ch(\Z)$, and $De(\Z)$ respectively.

        \begin{definition} [Factorization over a DAG]
            Let $G = (\V, \E)$ be a DAG. A joint distribution $P(\vv)$ is said to \textit{factorize} according to $G$ if
            \[
            P(\vv) = \prod_{v \in \V} P\left(v \mid Pa(v)\right),
            \]
            where $Pa(v)$ denotes the set of parents of $V$ in $G$.

            If $G = (\V \sqcup \U, \E)$ includes unobserved variables $\U$, the distribution over observed variables is given by the marginal
            \[
                P(\vv) = \sum_{\uu} P(\vv, \uu).
            \]
        \end{definition}

        Directed edges $X \to Y$ in $G$ indicate $X$ is a direct cause of $Y$, so $G$ represents a hypothesized causal model underlying an observational distribution $P(\vv)$.
    \subsection{D-separation}
        Beyond factorization, a DAG can be understood through the conditional independencies it implies. These graphical independencies are often simpler to reason about than the factorization itself. The following examples illustrate the correspondence.
        \begin{example}
            Consider the graphs below:
            
            % https://q.uiver.app/#q=WzAsMjAsWzEsMCwiWCJdLFsyLDAsIloiXSxbMywwLCJZIl0sWzIsMSwiWiJdLFsxLDIsIlgiXSxbMywyLCJZIl0sWzEsMywiWCJdLFsyLDQsIlpfMSJdLFszLDMsIlkiXSxbMiw1LCJaXzIiXSxbNiwwLCJYXFxwZXJwXFwhXFwhXFwhXFxwZXJwIFkgXFxtaWQgWiJdLFs1LDAsIlggXFxub3RcXHBlcnBcXCFcXCFcXCFcXHBlcnAgWSJdLFs1LDIsIlggXFxub3RcXHBlcnBcXCFcXCFcXCFcXHBlcnAgWSJdLFs2LDIsIlhcXHBlcnBcXCFcXCFcXCFcXHBlcnAgWSBcXG1pZCBaIl0sWzYsNCwiWCBcXG5vdFxccGVycFxcIVxcIVxcIVxccGVycCBZIFxcbWlkIFpfMSJdLFs1LDQsIlggXFxwZXJwXFwhXFwhXFwhXFxwZXJwIFkiXSxbNyw0LCJYIFxcbm90XFxwZXJwXFwhXFwhXFwhXFxwZXJwIFkgXFxtaWQgWl8yIl0sWzAsMCwiXFx0ZXh0eyhhKX0iXSxbMCwyLCJcXHRleHR7KGIpfSJdLFswLDQsIlxcdGV4dHsoYyl9Il0sWzAsMV0sWzEsMl0sWzMsNF0sWzMsNV0sWzYsN10sWzgsN10sWzcsOV1d
            \begin{center}
            \scalebox{0.8}{
            \begin{tikzcd}[ampersand replacement=\&]
                {\text{(a)}} \& X \& Z \& Y \&\& {X \not\perp\!\!\!\perp Y} \& {X\perp\!\!\!\perp Y \mid Z} \\
                \&\& Z \\
                {\text{(b)}} \& X \&\& Y \&\& {X \not\perp\!\!\!\perp Y} \& {X\perp\!\!\!\perp Y \mid Z} \\
                \& X \&\& Y \\
                {\text{(c)}} \&\& {Z_1} \&\&\& {X \perp\!\!\!\perp Y} \& {X \not\perp\!\!\!\perp Y \mid Z_1} \& {X \not\perp\!\!\!\perp Y \mid Z_2} \\
                \&\& {Z_2}
                \arrow[from=1-2, to=1-3]
                \arrow[from=1-3, to=1-4]
                \arrow[from=2-3, to=3-2]
                \arrow[from=2-3, to=3-4]
                \arrow[from=4-2, to=5-3]
                \arrow[from=4-4, to=5-3]
                \arrow[from=5-3, to=6-3]
            \end{tikzcd}}
            \end{center}
            \captionof{figure}{}\label{fig: d-separation}
            \begin{enumerate}
                \item [(a)] There is a direct causal chain $X \rightarrow Z \rightarrow Y$, so $X \not \indep Y$. However, $X$ influences $Y$ only through their $Z$. Once $Z$ is known, learning $X$ provides no additional information about $Y$. Thus, $X \indep Y \mid Z$. 
                \item [(b)] $X$ and $Y$ share a common parent $Z$. Knowing $X$ gives us partial information about $Y$ through their common ancestry, like siblings who share genetic material, so $X \not \indep Y$. But if we already know the parent $Z$, then $X$ tells us nothing further about $Y$, and again $X \indep Y \mid Z$. 
                \item [(c)] $X$ and $Y$ are the parents of a common child $Z_1$. Just as two parents typically do not share genetic material directly, $X$ and $Y$ are independent. Yet once we condition on their child $Z_1$ (or even a grandchild $Z_2$), knowing one parent gives us information about the other through what they jointly contribute to the descendant. Thus $X \not\indep Y \mid Z_1$ (and likewise given $Z_2$), even though $X \indep Y$.
            \end{enumerate}
        \end{example}
        These examples illustrate the three fundamental ways that paths can transmit or block dependence between variables. In larger graphs, deducing independencies is simply an extension of these base rules. We now formalize this idea in the following definitions:
        
        \begin{definition}
            Let $V$ be an interior node on a path $\pi$. $V$ is a
            \begin{enumerate}
                \item \textit{collider} if both edges adjacent to $V$ on $\pi$ have arrowheads pointing towards $V$, i.e., $\cdots \rightarrow V \leftarrow \cdots$.
                \item \textit{non-collider} if $V$ is not a collider on $\pi$, i.e., $\cdots \leftarrow V \leftarrow \cdots$, $\cdots \leftarrow V \rightarrow \cdots$, or $\cdots \rightarrow V \rightarrow \cdots$.
            \end{enumerate}
        \end{definition}
        
        \begin{definition} [Open / Closed Paths]
            Let $G = (\V \sqcup \U, \E)$, $\Z \subset \V$. A path $\pi$ from $X$ to $Y$ is \textit{open} given (or conditioned on) $\Z$ if
            \begin{enumerate}
                \item All non-colliders in $\pi$ are not in $\Z$
                \item All colliders in $\pi$ are in $An(\Z)$
            \end{enumerate}
            A path is \textit{closed} if it is not open.
        \end{definition}

        The following definition is the graphical criterion for independence given by a probability distribution $P$'s graph.
        \begin{definition} [D-separation]
            Let $G = (\V \sqcup \U, \E)$ and $\X, \Y, \Z \subset \V$. We say that $\X$ and $\Y$ are \textit{d-separated} by $\Z$ (denoted $\X \indep_G \Y \mid \Z$) if there is no path from any node in $\X$ to any node in $\Y$ that is open given $\Z$. 
            
            We say that $\X$ and $\Y$ are \textit{d-connected} if they are not d-separated.
        \end{definition}
        
        Intuitively, d-separation tells us when conditioning on $\Z$ blocks all paths through which probabilistic dependence between $\X$ and $\Y$ could arise.

        \begin{remark}
            We often drop the subscript $G$ in $\indep_G$ when the corresponding graph is unambiguous.
        \end{remark}
        
        We then associate a graph with a probability distribution as follows:
        \begin{definition}[Independence Map]
            A DAG $G$ is an \emph{independence map} of a probability distribution $P$ if for all disjoint sets $\X, \Y, \Z$:
            \[
            (\X \indep_G \Y \mid \Z) \implies (\X \indep_P \Y \mid \Z).
            \]
            In other words, every d-separation in $G$ implies a corresponding conditional independence in $P$. The conditional independencies encoded by $G$ therefore provide a “map” of the independence structure of $P$, hence the term independence map.
        \end{definition}
        
        \begin{remark}
            The converse is not guaranteed: $\X \not\!\indep_G \Y \mid \Z$ does not necessarily imply $\X \not\!\indep_P \Y \mid \Z$ if $G$ is an independence map of $P$. When this occurs, we say that $P$ is \emph{unfaithful} to $G$.
        \end{remark}

        \begin{theorem}[Equivalence of Factorization and D-Separation] \label{thm: factorization == i-map}
            For any DAG $G$ and probability distribution $P$, the following are equivalent:
            \begin{enumerate}
                \item $P$ factorizes according to $G$.
                \item $G$ is an independence map of $P$.
            \end{enumerate}
        \end{theorem}
        This theorem shows that reasoning about conditional independencies via d-separation is equivalent to reasoning via factorization: the DAG captures both the algebraic structure of $P$ and its independence relationships.

        \begin{remark}
            \Cref{thm: factorization == i-map} may be false if $G$ is not a DAG, as in undirected graphical models.
        \end{remark}
    \subsection{Interventions} 
        The central object of causal inference is not the observed distribution $P(\vv)$ itself, but the \textit{interventional distribution}—how this distribution would change under external manipulation of the system. Formally,
        
        \begin{definition}[Interventional Distribution]
            The \textit{interventional distribution}
            \[
            P(\vv \mid \mathrm{do}(\X=\x^*)) 
            = \sum_{\uu}\prod_{v \in (\V \sqcup \U) \setminus \X} P(v \mid Pa(v)) \cdot \mathbbm{1}_{\x=\x^*}.
            \]
            is obtained by externally fixing $\X$ to $\x^*$, thereby removing all causal influence of the parents of $\X$.

            Graphically, this corresponds to deleting all incoming edges into $\X$.
        \end{definition}

        \begin{definition}[Causal Effect]
            For a set of outcome variables $\Y$, the \textit{causal effect} of $\X$ on $\Y$ is the marginal
            \[
            P(\y \mid \mathrm{do}(\X=\x^*)) = \sum_{(\vv\sqcup \uu) \setminus \y } P(\vv\setminus \x \mid \mathrm{do}(\X=\x^*)).
            \]
        \end{definition}
        % It is essential to distinguish intervention from conditioning. Conditioning on $\X=\x^*$ restricts attention to observational units for which $X$ happened to take the value $\x^*$, leaving the underlying causal mechanisms unchanged. An intervention, by contrast, actively enforces $\X=\x^*$, severing the causal influence of its parents and generally producing a different distribution.

        In many real-world settings, direct experimental interventions\textemdash where one externally fixes an independent variable $X$ to observe its causal effect on a dependent variable $Y$\textemdash can be morally or financially infeasible. For instance, one cannot induce stress in pregnant mothers to observe its effects on fetal development.  In such cases, purely observational data may be the only source of information available.

        The central challenge of causal inference is then to determine when and how interventional distributions can be \textit{identified} solely from observational data, given assumptions encoded by a causal graph. This paper is concerned with understanding the conditions under which such identification is possible.

\section{Heuristic Derivations} \label{apx: heuristic_derivations}
    We give a brief outline of the counterfactual intuition which guided the generalized front-door criterion stated in \Cref{thm: main}. We aim to convey the reasoning behind the criterion, without presenting full technical derivations.

    Our approach was inspired by Shpitser's \cite{SHPITSER10} generalization of the back-door criterion, which established an equivalence to the counterfactual condition $\Y_\X \indep \X \mid \Z$ (known as conditional ignorability).

    We began by analyzing the counterfactual proof of Pearl's front-door criterion \cite[Section~7.3.2]{PEARL09}. There, condition (3) of \Cref{thm: fdc} helps establish the independence $\Z_{\X} \indep \Y_{\Z} \mid \X$, which is used to derive the key equality $P(\Y_\z \mid \x, \Z_{\x^*} = \z) = P(\y \mid \x, \z)$. We hypothesized that this specific independence was sufficient but not necessary to produce the equality. We conjectured a weaker set of counterfactual independence conditions could be found, corresponding to a weaker graphical criterion.
    
    Guided by an analysis of examples that violate the original front-door criterion yet for which the front-door functional is valid, we identified a vaguely natural partition of the set $\Z$ into children of $\X$ ($\Z^{ch}$) and non-children ($\Z^{nch}$). We then sought graphical conditions under which the independence conditions $\Y_{\Z} \indep \Z^{ch}_{\X} \mid \Z, \X$ and $\Y_{\Z} \indep \Z^{ch} \mid \Z_{\X}, \X$ would hold. Analysis using Twin Network Graphs \cite{SHPITSER10,SHPITSER08} yielded a first-draft graphical criterion slightly stronger than the final form presented in Theorem \ref{thm: main}. Culling the extraneous conditions was done by observing several examples in which they were unnecessary for the front-door functional to hold.

    Once a (hypothetically) correct set of graphical conditions was established, the completeness of do-calculus \cite{HUANG2006b} guaranteed that a do-calculus derivation of the front-door functional must exist; our proof in \Cref{apx: pf_main} provides this derivation, which we deduced by mirroring the structure of the counterfactual argument. 
   
\end{document}